\documentclass[12pt,a4paper,fleqn]{article}
\usepackage{latexsym}
\usepackage{amsmath}
\usepackage{amsthm}
\usepackage{amssymb}
\usepackage{amsfonts}
\usepackage{mathrsfs}
\usepackage{enumerate}
\usepackage{graphicx}
\usepackage{xcolor}
\usepackage{rotating}
\usepackage{caption}
\usepackage{longtable}
\newtheorem{theorem}{Theorem}[section]
\newtheorem*{theorem*}{Theorem}
\newtheorem{lemma}[theorem]{Lemma}
\newtheorem*{lemma*}{Lemma}
\newtheorem{corollary}[theorem]{Corollary}
\newtheorem*{corollary*}{Corollary}
\newtheorem{proposition}[theorem]{Proposition}
\newtheorem*{proposition*}{Proposition}

{\theoremstyle{definition}
\newtheorem{definition}[theorem]{Definition}
\newtheorem*{definition*}{Definition}
\newtheorem{example}[theorem]{Example}
\newtheorem*{example*}{Example}
\newtheorem*{examples*}{Examples}
\newtheorem{remark}[theorem]{Remark}
\newtheorem*{remark*}{Remark}
\newtheorem{result}{Result}

\newtheorem*{question*}{Question}

\newtheorem*{problem*}{Problem}

\newtheorem*{notation*}{Notation}

\newtheorem*{note*}{Note}

\newtheorem*{algorithm*}{Algorithm}
}
\renewcommand{\arraystretch}{0.75}

\newcommand{\Colon}{\:\mathbin{:}\:}

\newcommand{\fk}[1]{{\mathfrak #1}}

\newcommand{\proofblock}{\hspace{\fill}\fbox{\rule{0pt}{3pt}\rule{3pt}{0pt}}}
\newcommand{\LEQ}{\leqslant}
\newcommand{\GEQ}{\geqslant}

\let\unlhd\trianglelefteq
\newcounter{xcoord}
\newcounter{ycoord}

\title{Minimal determining sets for certain $W$-graph ideals}
\author{T. P. McDonough%
\thanks{%
\textit{Department of Mathematics,
Aberystwyth University,
Aberystwyth SY23 3BZ,
United Kingdom.}
{E-mail: tpd@aber.ac.uk}
}
\ and\ \
C. A. Pallikaros%
\thanks{%
\textit{Department of Mathematics and Statistics,
University of Cyprus,
P.O.Box 20537,
1678 Nicosia,
Cyprus.}
{E-mail: pallikar@ucy.ac.cy}
}
}
\date{December 12, 2021} 
\begin{document}
\maketitle
\begin{abstract}
We consider Kazhdan-Lusztig cells of the symmetric group $S_n$ containing the longest element of a standard parabolic subgroup of $S_n$.
Extending some  of the ideas in [{Beitr{\"a}ge} zur Algebra und Geometrie, \textbf{59} (2018), no.~3, 523--547] and [Journal of Algebra and Its Applications, {\bf20} (2021), no.~10, 2150181], we determine the rim of some additional families  of cells and also of certain induced unions of cells.
These rims provide minimal determining sets for certain $W$-graph ideals introduced in [Journal of Algebra, {\bf361} (2012), 188--212].
\\[1.5ex]
\end{abstract}
\section{Introduction}

In~\cite{KLu79} Kazhdan and Lusztig introduced the left cells, the right cells and the two-sided cells of a Coxeter group $W$ as a means of investigating the representation theory of $W$ and its associated Hecke algebra~$\mathcal H$.
It is also shown in~\cite{KLu79} that in the case $W=S_n$ the Robinson-Schensted correspondence gives a combinatorial description of the Kazhdan-Lusztig cells.
However, this does not lead to some straightforward way of obtaining reduced forms for the elements in these cells.

The present paper is a continuation of the work in~\cite{MPa08,MPa15,MPa17,MPa21} and it is concerned with the problem of determining reduced expressions for all the elements in a given cell and also in certain induced unions of cells.
(See~\cite{BVo83}, \cite{Roi98} and~\cite{Gec03} for the induction of Kazhdan-Lusztig cells.)
The focus is on (right) cells containing the longest element of a Young subgroup of $S_n$ and also on the union of cells obtained by inducing such cells to $S_{n+1}$.
By extending certain ideas in~\cite{MPa17,MPa21} we are able to determine the rim of some additional families of Kazhdan-Lusztig cells and also of the corresponding induced union of cells.
As a result, reduced forms for all the elements in these subsets of $S_n$ can be obtained directly.

Motivated by the $W$-graph structure with which the regular representation of $\mathcal H$ is endowed in~\cite{KLu79}, Howlett and Nguyen~\cite{HN12} introduced a notion of $W$-graph ideal in $W$.
The work in this paper is closely connected with the work in~\cite{HN12,Ngu12,Ngu15,HN16}, as the elements of the rims obtained for the various subsets of $W=S_n$ we investigate in fact provide minimal determining sets for certain (right) $W$-graph ideals.

The paper is organized as follows:
In Section~\ref{sec:2c} our aim is to investigate the connection between right ideals in $W$ and root systems and, via this approach, in Proposition~\ref{prop:3.12} we show how the minimal determining set of the right ideal $Z\mathfrak X_J$ in $W$ can be obtained explicitly given the minimal determining set of a right ideal $Z$ in $W_J$.
(By $W_J$ we denote a standard parabolic subgroup of $W$ and by $\mathfrak X_J$ the set of distinguished right coset representatives of $W_J$ in $W$.)
In Section~\ref{subsec_sym_gr} we recall some background on ordered $k$-paths and admissible diagrams from~\cite{MPa17,MPa21}, while Section~\ref{sec:3a} is mainly concerned with the identification and investigation of certain key ordered $k$-paths having an admissible diagram as their support.
Finally, in Section~\ref{sec5}, using the ideas developed in the earlier parts of the paper, we obtain explicit descriptions for the minimal determining sets of certain $W$-graph ideals in $S_n$ corresponding to Kazhdan-Lusztig cells and induced unions of such cells (see Theorem~\ref{thm:3.13a} and Remark~\ref{Rem5.2}).

\begin{section}
{Root systems and ideals in $W$}
\label{sec:2c}

For a Coxeter system $(W,S)$, Kazhdan and Lusztig \cite{KLu79}
introduced the notion of a $W$-graph and used this notion to define  three preorders $\LEQ_L$, $\LEQ_R$ and $\LEQ_{LR}$, with
corresponding equivalence relations $\sim_L$, $\sim_R$ and $\sim_{LR}$,
whose equivalence classes are called
\emph{left cells}, \emph{right cells} and \emph{two-sided cells},
respectively.
Each cell of $W$ provides a representation of $W$, with the $C$-basis of the Hecke algebra $\mathcal H$ of $(W,S)$ playing an important role in the construction of this representation; see~\mbox{\cite[\S~1]{KLu79}}.
The $C$-basis equips the regular representation of $\mathcal H$ with a $W$-graph structure, one of the facts playing an important role in~\cite{KLu79}.

For the rest of this section we assume that $(W,S)$ is a Coxeter system with $W$ finite.
Also let $J\subseteq S$. Then $(W_J,J)$ is a Coxeter system, where $W_J=\langle J\rangle$ denotes the standard parabolic subgroup determined by a subset $J$ of $S$.
We denote by $w_J$ the longest element of $W_J$ and by $\mathfrak{X}_J$ the set of minimum length elements in the right cosets of $W_J$ in $W$ (the distinguished right coset representatives).
Recall the \emph{prefix relation} on the elements of~$W$:
if $x,y\in W$ we say that $x$ is a \emph{prefix} of $y$ if $y$ has a \emph{reduced form} beginning with a reduced form for $x$.
We then have that $\mathfrak X_J$ is the set of prefixes of $d_J$ where $d_J$ is the longest element of $\mathfrak X_J$ (see~\cite[Lemma~2.2.1]{GPf00}).
Also recall that the right cell containing $w_J$ is contained in $w_J\mathfrak{X}_J$ (see~\cite[5.26.1]{Lus84b}).

A \emph{right ideal} in $W$ is a subset in $W$ which is closed under the taking of prefixes.
Given a right ideal $\mathscr{I}$ in $W$, we call the set $Y(\mathscr{I})=\{x\in\mathscr{I}\colon x$ is not the prefix of any other $y\in\mathscr{I}\}$ the \emph{minimal determining set} for $\mathscr{I}$ since knowledge of $Y(\mathscr{I})$ leads directly to $\mathscr{I}$ by taking all prefixes.

The $W$-graphs introduced in~\cite{KLu79} encode in a very concise way the structure of certain representations of $\mathcal H$.
Motivated by the ideas in~\cite{KLu79}, Howlett and Nguyen in~\cite{HN12} introduced the notion of a $W$-graph ideal in $W$ and produced, for any such ideal, a $W$-graph via an algorithm like the Kazhdan-Lusztig algorithm.
A $W$-graph ideal is an ideal in $W$ with the additional property that it admits a module structure in a very particular way (see~\cite[Definition~5.1]{HN12}).
In particular, the subsets $Z_J=\{d\in\mathfrak X_J\colon w_Jd\sim_Rw_J\}$ and $Z\mathfrak{X}_J$ (where $Z$ is a $W$-graph ideal with respect to $J$ in $W_J$) of $W$
are $W$-graph ideals (with respect to $J$) in $W$ (see~\cite[Theorem~5.4]{Ngu12} and \cite[Theorem~9.2]{HN12}).

As a consequence, if $(\hat W,\hat S)$ is a Coxeter system with $S\subseteq\hat S$ and $\hat{\mathfrak X}$ is the set of distinguished right coset representatives of $W$ in $\hat W$, the set $Z_J\hat{\mathfrak X}$ is a $W$-graph ideal with respect to $J$ in $\hat W$.
Note that the set $w_J Z_J\hat{\mathfrak X}$ is the union of Kazhdan-Lusztig cells in $\hat W$ obtained from inducing to $\hat W$ the cell containing $w_J$ in $W$.

Considering the connection between right ideals and root systems, our aim in this section is to relate explicitly, via this approach, the minimal determining sets of the right ideals $Z$ and $Z\mathfrak X_J$ of $W_J$ and $W$ respectively (see Proposition~\ref{prop:3.12}).

Let $\boldsymbol\Phi$ be the root system corresponding to $(W,S)$ and let
$V=\langle\boldsymbol\Phi\rangle$;
let $\boldsymbol\Sigma$ be a set of fundamental roots for $\boldsymbol\Phi$,
let $\boldsymbol\Phi^{+}$ be the positive roots in $\boldsymbol\Phi$ and
$\boldsymbol\Phi^{-}=-\boldsymbol\Phi^{+}$ the negative roots.
Also let $\boldsymbol\Phi_J$ be the subsystem of $\boldsymbol\Phi$
corresponding to the subsystem $(W_J,J)$ of $(W,S)$.

For each $s\in S$, let $\alpha_s$ be the root corresponding to $s$
and let $\rho_s$ be the reflection corresponding to $s$.
So $\alpha\rho_s
=\alpha-\frac{2(\alpha,\alpha_s)}{(\alpha_s,\alpha_s)}\alpha_s$ for each $\alpha\in V$ (we suppose that $GL(V)$ acts on $V$ on the right).
There is an  injective group homomorphism
$\rho\colon W\rightarrow GL(V)$ defined by $x\mapsto \rho_x=\rho_{u_1}\cdots\rho_{u_r}$ where
$u_1,\ldots u_r\in S$ and $u_1\cdots u_r$
is any reduced word for $x\in W$.
It will be convenient to write  $vx$ for $v\rho_x$ where $v\in V$ and
$x\in W$.

For any $\beta\in V$, $\beta\neq0$, the reflection $\rho_{\beta}$
is given by $\alpha\rho_{\beta}
=\alpha-\frac{2(\alpha,\beta)}{(\beta,\beta)}\beta$
for each $\alpha\in V$.
So $\rho_s=\rho_{\alpha_s}$.
For $x\in W$, let $N(x)$ be the set of positive roots of $W$ which
are mapped by $x$ to negative roots.
That is, $N(x)
=\boldsymbol\Phi^+\cap\boldsymbol\Phi^-x^{-1}$.
\par

Below we collect some basic results on roots and the length function.

\begin{result}[{\cite[Theorem~p.111, Lemma~p.116]{Hum90}}]\label{thm1} 
Let $x\in W$, $s\in S$ and $\alpha,\beta\in\boldsymbol\Phi$.

(i) If $l(sx)>l(x)$ then $\alpha_s x\in\boldsymbol\Phi^{+}$.
If $l(sx)<l(x)$ then $\alpha_s x\in\boldsymbol\Phi^{-}$.

(ii) If $\alpha x=\beta$ then $x^{-1} \rho_{\alpha} x =\rho_{\beta}$.
%
\end{result}

\begin{result}[{\cite[Proposition~1.3.5]{GPf00}}]
\label{thm4} 
Let $x\in W$ and write $x=u_1\cdots u_r$ where $u_1,\ldots,u_r\in S$ and $r=l(x)$;
that is, $x$ is written as a reduced word.
Let $\beta_i=\alpha_{u_i}u_{i-1}\cdots u_1$ for $1\LEQ i\LEQ r$, interpreting $\beta_1$ to be $\alpha_{u_1}$.
Then $|N(x)|=l(x)$ and $N(x)=\{\beta_i\colon 1\LEQ i\LEQ r\}$.
\end{result}

\begin{corollary}
\label{cor5}
Let $x\in W$ and let $x'$ be a prefix of $x$.
Then $N(x')\subseteq N(x)$.
\end{corollary}
\begin{proof}
We may write $x'=u_1\cdots u_p$ and $x=u_1\cdots u_p\cdots u_r$ where
$u_1,\ldots,u_r\in S$, $l(x')=p$ and $l(x)=r$.
Using the notation in Result~\ref{thm4}, we get
$N(x)=\{\beta_i\colon 1\LEQ i\LEQ r\}$ and
$N(x')=\{\beta_i\colon 1\LEQ i\LEQ p\}$.
Hence, $N(x')\subseteq N(x)$.
\end{proof}
\begin{corollary}
\label{cor6}
Let $x\in W$, $s\in S$ and suppose that $l(sx)<l(x)$.
Then $N(x)=\left(N(sx)\right)s\cup\{\alpha_s\}$.
\end{corollary}
\begin{proof}
Write $x=u_1\cdots u_r$ where $u_1,\ldots,u_r\in S$, $u_1=s$ and
$l(x)=r$.
Let $\gamma_i=\alpha_{u_i}u_{i-1}\cdots u_2$, $2\LEQ i\LEQ r$.
By Result~\ref{thm4}, $N(sx)=\{\gamma_i\colon 2\LEQ i\LEQ r\}$.
Since $\beta_i=\gamma_i s$, $2\LEQ i\LEQ r$, and
$\beta_1=\alpha_{u_1}=\alpha_s$, we get the desired result.
\end{proof}
\begin{corollary}
\label{cor7}
Let $x,x'\in W$ and suppose that $N(x')\subseteq N(x)$.
Then $x'$ is a prefix of $x$.
\end{corollary}
\begin{proof}
Write $x=u_1\cdots u_r$ and $x'=u_1'\cdots u_p'$ where
$u_1,\ldots,u_r,u_1',\ldots, u_p'\in S$, $l(x)=r$ and $l(x')=p$.
If $l(x')=0$ then $x'=1$ which is trivially a prefix of $x$.
So we may suppose that $l(x')>0$.
\par
Let $\beta_i=\alpha_{u_i}u_{i-1}\cdots u_1$
for $1\LEQ i\LEQ r$ and let
$\gamma_i=\alpha_{u_i'}u_{i-1}'\cdots u_1'$, $1\LEQ i\LEQ p$.
By Result~\ref{thm4}, $N({x})=\{\beta_i\colon 1\LEQ i\LEQ r\}$ and
$N(x')=\{\gamma_i\colon 1\LEQ i\LEQ p\}$.
Since $\alpha_{u_1'}=\gamma_1\in N(x')\subseteq N(x)$,
$\alpha_{u_1'}=\beta_i$ for some $1\LEQ i\LEQ r$.
That is, $\alpha_{u_1'}=\alpha_{u_i}u_{i-1}\cdots u_1$.
By Result~\ref{thm1}(ii),
$u_1\cdots u_{i-1}\cdot u_i\cdot u_{i-1}\cdots u_1=u_1'$.
So $u_1\cdots u_i=u_1'\cdot u_1\cdots u_{i-1}$.
Replacing
$u_1\cdots u_i$ in the original reduced word for $x$ by
$u_1'\cdot u_1\cdots u_{i-1}$, we get a reduced word for
$x$ starting with $u_1'$.
That is, $l(u_1'x)<l(x)$.
By Corollary~\ref{cor6},
$N(x)=\left(N(u_1'x)\right)u_1'\cup\{\alpha_{u_1'}\}$.
\par
Also by Corollary~\ref{cor6},
$N(x')=\left(N(u_1'x')\right)u_1'\cup\{\alpha_{u_1'}\}$
since $l(u_1'x')<l(x')$.
Hence, $N(u_1'x')\subseteq N(u_1'x)$.
By induction, $u_1'x'$ is a prefix of $u_1'x$.
Hence, $x'$ is a prefix of~$x$.
\end{proof}
\par
Combining Corollaries~\ref{cor5} and~\ref{cor7}, we get the following
proposition.
\begin{proposition}
\label{cor8}
Let $x,x'\in W$. Then $N(x')\subseteq N(x)$ if, and only if,
$x'$ is a prefix of $x$.
In particular, $N(x')=N(x)$ if, and only if, $x'=x$.
\end{proposition}

\begin{proposition}
\label{prop12}
Let $x=ud\in W$ with $u\in W_J$ and $d\in\fk{X}_J$.
Then $N(x)=N(u)\cup N(d)u^{-1}$ and
$N(u)\cap N(d)u^{-1}=\varnothing$.
\end{proposition}
\begin{proof}
Let $u=u_1\cdots u_q$ and $d=u_{q+1}\ldots u_r$, where $u_i\in S$ for
$1\LEQ i\LEQ r$, be reduced words for $u\in W_J$ and $d\in\fk{X}_J$.
Then $x=u_1\cdots u_r$ is also a reduced word.
Let $\beta_i=\alpha_{u_i}u_{i-1}\cdots u_1$ for $1\LEQ i\LEQ r$.
By Result~\ref{thm4},
$N(x)=\{\beta_i\colon 1\LEQ i\LEQ r\}$ and
$N(u)=\{\beta_i\colon 1\LEQ i\LEQ q\}$.
Let $\gamma_i=\alpha_{u_i}u_{i-1}\cdots u_{q+1}$ for $q+1\LEQ i\LEQ r$.
Then $\gamma_iu^{-1}=\beta_i$ for $q+1\LEQ i\LEQ r$ and,
again by Result~\ref{thm4},
$N(d)=\{\gamma_i\colon q+1\LEQ i\LEQ r\}$.
Thus, $N(x)$ is the disjoint union of $N(u)$ and $N(d)u^{-1}$.
\end{proof}
\begin{proposition}
\label{prop13}
Let $\boldsymbol\Phi^+_J=\boldsymbol\Phi_J\cap\boldsymbol\Phi^+$.
If $v\in W_J$ then $(\boldsymbol\Phi^+-\boldsymbol\Phi^+_J)v\subseteq\boldsymbol\Phi^+-\boldsymbol\Phi^+_J$.
\end{proposition}
\begin{proof}
Recall that $N(v)$ is the set of positive roots of $W$ which
are mapped by $v$ to negative roots;
that is, $N(v)
=\boldsymbol\Phi^+\cap\boldsymbol\Phi^-v^{-1}$.
Moreover, $l(v)=|N(v)|=|\boldsymbol\Phi^+\cap\boldsymbol\Phi^-v^{-1}|$.
\par
Let $\boldsymbol\Psi=\boldsymbol\Phi_J$, let $\boldsymbol\Psi^+$ be the positive roots of $\boldsymbol\Phi_J$
and $\boldsymbol\Psi^-=-\boldsymbol\Psi^+$ be the corresponding negative roots.
Let $N_J(v)$ be the set of positive roots of $W_J$ which
are mapped by $v$ to negative roots and let $l_J$ be the length function of $(W_J,J)$.
So $N_J(v)
=\boldsymbol\Psi^+\cap\boldsymbol\Psi^-v^{-1}$ and
$l_J(v)=|N_J(v)|=|\boldsymbol\Psi^+\cap\boldsymbol\Psi^-v^{-1}|$.
Since $\boldsymbol\Psi\subseteq\boldsymbol\Phi$,
$\boldsymbol\Psi^+\subseteq\boldsymbol\Phi^+$ and $\boldsymbol\Psi^-\subseteq\boldsymbol\Phi^-$,
we have
$N_J(v)=\boldsymbol\Psi^+\cap\boldsymbol\Psi^-v^{-1}\subseteq\boldsymbol\Phi^+\cap\boldsymbol\Phi^-v^{-1}=N(v)$,
As $|N_J(v)|=l_J(v)=l(v)=|N(v)|$, it follows that $N_J(v)=N(v)$.
So $\boldsymbol\Psi^+\cap\boldsymbol\Psi^-v^{-1}=\boldsymbol\Phi^+\cap\boldsymbol\Phi^-v^{-1}$,
and $\boldsymbol\Psi^+v\cap\boldsymbol\Psi^-=\boldsymbol\Phi^+v\cap\boldsymbol\Phi^-$.
\par
Now let $\alpha\in\boldsymbol\Phi^+-\boldsymbol\Psi^+$.
Then $\alpha v\in\boldsymbol\Phi^+v$ and $\alpha v\notin\boldsymbol\Psi^+v$.
If $\alpha v\in\boldsymbol\Phi^-$,
then $\alpha v\in\boldsymbol\Phi^+v\cap\boldsymbol\Phi^-=\boldsymbol\Psi^+v\cap\boldsymbol\Psi^-$,
contrary to $\alpha v\notin\boldsymbol\Psi^+ v$.
Hence $\alpha v\notin\boldsymbol\Phi^-$.
As $\alpha\notin\boldsymbol\Psi$, $\alpha v\notin\boldsymbol\Psi$.
So $\alpha v\in\boldsymbol\Phi^+-\boldsymbol\Psi=\boldsymbol\Phi^+-\boldsymbol\Psi^+$.
\end{proof}
\begin{proposition}
\label{prop14}
$N(d_J)v=N(d_J)=\boldsymbol\Phi^+-\boldsymbol\Phi^+_J$ for all $v\in W_J$.
In particular, $N(d)v\subseteq\boldsymbol\Phi^+-\boldsymbol\Phi^+_J$ for all $d\in\mathfrak X_J$ and $v\in W_J$.
\end{proposition}
\begin{proof}
We denote by $w_S$  the element of maximum length in $W$.
Then $w_S^2=1$ and $N(w_S)=\boldsymbol\Phi^+$ (see~\cite[p.~27]{GPf00}).
\par
Continuing with the notation of Proposition~\ref{prop13}, we also get
$N(w_J)=\boldsymbol\Psi^+$.
We can write $w_S=w_Jd_J$.
By Proposition~\ref{prop12},
$\boldsymbol\Phi^+=N(w_S)=N(w_Jd_J)=N(w_J)\cup N(d_J)w_J^{-1}$
and $N(w_J)\cap N(d_J)w_J^{-1}=\varnothing$.
So $N(d_J)w_J^{-1}=\boldsymbol\Phi^+-\boldsymbol\Psi^+$.
By Proposition~\ref{prop13},
$(\boldsymbol\Phi^+-\boldsymbol\Psi^+)w_J\subseteq\boldsymbol\Phi^+-\boldsymbol\Psi^+$.
Hence, $N(d_J)\subseteq\boldsymbol\Phi^+-\boldsymbol\Psi^+$.
Comparing the sizes of these sets, we get
$N(d_J)=\boldsymbol\Phi^+-\boldsymbol\Psi^+$.
\par
Now let $v\in W_J$.
From Proposition~\ref{prop13},
$N(d_J)v=(\boldsymbol\Phi^+-\boldsymbol\Psi^+)v\subseteq\boldsymbol\Phi^+-\boldsymbol\Psi^+=N(d_J)$.
Again comparing sizes, we get $N(d_J)v=N(d_J)$.
\end{proof}

\begin{corollary}\label{cor_3.11} Suppose that $d\in\mathfrak X_J$ and $u_1, u_2\in W_J$ with $u_1$ a prefix of $u_2$.
Then $u_1d$ is a prefix of $u_2d_J$.
\end{corollary}
\begin{proof}
From Proposition~\ref{prop12}, $N(u_2d_J)=N(u_2)\cup N(d_J)u_2^{-1}$.
Hence $N(u_2d_J)=N(u_2)\cup(\boldsymbol\Phi^+-\boldsymbol\Phi_J^+)$ by Proposition~\ref{prop14}.
Again, by Proposition~\ref{prop12}, $N(u_1d)=N(u_1)\cup N(d)u_1^{-1}$.
But $N(u_1)\subseteq N(u_2)$ and $N(d)\subseteq N(d_J)$ from Proposition~\ref{cor8}.
In view of Propositions~\ref{cor8} and~\ref{prop13}, it follows that $N(d)u_1^{-1}\subseteq N(d_J)u_1^{-1}=\boldsymbol\Phi^+-\boldsymbol\Phi_J^+$.
This leads to $N(u_1d)=N(u_1)\cup N(d)u_1^{-1}\subseteq N(u_2)\cup(\boldsymbol\Phi^+-\boldsymbol\Phi_J^+)=N(u_2d_J)$.
The desired result is now immediate from Proposition~\ref{cor8}.
\end{proof}

\begin{result}[{Compare with \cite[Lemma~9.1]{HN12}}]
\label{prop9}
Let $x=ud\in W$ with $u\in W_J$ and $d\in\fk{X}_J$,
and let $x'=u'd'$ be a prefix of $x$ with $u'\in W_J$ and
$d'\in\fk{X}_J$.
Then $u'$ is a prefix of $u$ and $d'$ is a prefix of $d$.
In particular, if $Z$ is a right ideal in $W_J$, then $Z\fk{X}_J$ is a right ideal in $W$.
\end{result}

\begin{remark}
\label{rem11}
The converse of Result~\ref{prop9} is false.
Let $W=S_3=\langle S\rangle$ where $S=\{s_1,s_2\}$ with
$s_1=(1,2)$ and $s_2=(2,3)$.
Let $J=\{s_1\}$.
Then $W_J=\{1,s_1\}$ and $\fk{X}_J=\{1,s_2,s_2s_1\}$.
Consider $x=s_1s_2$.
Then $u=s_1$ and $d=s_2$.
Let $x'=s_2=u'd'$ where $u'=1$ and $d'=s_2$.
Then $u'$ is a prefix of $u$ and $d'$ is a prefix of $d$.
However, $(x')^{-1}x=s_2s_1s_2$ has length 3.
So $x'$ is not a prefix of $x$.
In Corollary~\ref{cor_3.11} we have seen that the converse of Result~\ref{prop9} is true in the special case $d=d_J$.
\end{remark}

\begin{proposition}\label{prop:3.12}
Let $Y$ be the minimal determining set of the right ideal $Z$ of $W_J$.
Then $Yd_J=\{xd_J\colon x\in Y\}$ is the minimal determining set of the right ideal $Z\mathfrak X_J$ of $W$.
\end{proposition}
\begin{proof}
Let $t\in Z\mathfrak X_J$.
Then $t=zd$ for some $z\in Z$ and $d\in\mathfrak X_J$, with $z$ a prefix of $\hat x$ for some $\hat x\in Y$.
By Corollary~\ref{cor_3.11}, $t=zd$ is a prefix of $\hat xd_J$.
It follows that the set $Yd_J$ contains the minimal determining set of $Z\mathfrak X_J$.
In order to complete the proof it is enough to establish that $x_1d_J$ is not a prefix of $x_2d_J$ whenever $x_1,x_2\in Y$ ($x_1\ne x_2$).
Suppose, on the contrary, that $x_1$, $x_2$ are distinct elements of $Y$ and $x_1d_J$ is a prefix of $x_2d_J$.
By Proposition~\ref{cor8}, $N(x_1d_J)\subseteq N(x_2d_J)$.
But $N(x_1d_J)$ (resp., $N(x_2d_J)$) is the disjoint union of $N(x_1)$ (resp., $N(x_2)$) and $(\boldsymbol\Phi^+-\boldsymbol\Phi_J^+)$ in view of Propositions~\ref{prop12} and~\ref{prop14}.
It follows that $N(x_1)\subseteq N(x_2)$, that is, $x_1$ is a prefix of $x_2$, which is the desired contradiction.
\end{proof}

\section{Symmetric group background}\label{subsec_sym_gr}

For the rest of this paper we  focus on the symmetric group.
For the basic definitions and background concerning partitions, compositions, Young diagrams, Young tableaux and the  Robinson-Schensted
correspondence we refer to~\cite{Sagan}.

The symmetric group $S_n$ (acting on the right) on $\{1,\dots,n\}$ is a Coxeter group with Coxeter system
$(W,S)$ where $W=S_n$, $S=\{s_1,\ldots,s_{n-1}\}$, and $s_i$ is the
transposition $(i,i+1)$.

All our partitions and compositions will be assumed to be \emph{proper}
(that is, with no zero parts).
We use the notation $\lambda\vDash n$ (respectively, $\lambda\vdash n$)
to say that $\lambda$ is a composition (respectively, partition) of $n$.
If $\nu,\mu\vdash n$ with $\nu=(\nu_1, \ldots , \nu_r)$ and $\mu=(\mu_1,\ldots,\mu_s)$, write
$\nu\unlhd\mu$ if $r\GEQ s$ and
$\sum_{1\LEQ i\LEQ k}\nu_i \LEQ \sum_{1\LEQ i\LEQ k}\mu_i$, for all $k$ with $1\LEQ k\LEQ s$.
This is the dominance order of partitions (see~\cite[p.~58]{Sagan}).

Let $\lambda=(\lambda_1, \ldots , \lambda_r)$ be a \textit{composition}
of $n$ with $r$ \emph{parts}.
Recall that the \emph{conjugate} composition
$\lambda'=(\lambda_1',\ldots,\lambda_{r'}')$ of $\lambda$ is defined by
$\lambda'_i=\left|\{j\Colon 1\LEQ j\LEQ r\mbox{ and }
i\LEQ\lambda_j\}\right|$ for $1\LEQ i\LEQ r'$, where $r'$ is the maximum part of the composition $\lambda$.
It is immediate that $\lambda'$ is a partition of $n$ with $r'$ parts.
We also define the subset
$J(\lambda)$
\label{def:j_lam}
of $S$ to be
$S\backslash \{s_{\lambda_1},s_{\lambda_1+\lambda_2},\ldots,
s_{\lambda_1+\ldots+\lambda_{r-1}}\}$.
Thus, corresponding to the composition $\lambda$, there is a standard
parabolic subgroup of $W$, also known as a Young subgroup, whose Coxeter generator set is $J(\lambda)$.

It was shown in~\cite{KLu79} that in the case of the symmetric group $S_n$, the Robinson-Schensted
correspondence gives a combinatorial method of identifying the
Kazhdan-Lusztig cells. 
In describing the connection
between the Kazhdan-Lusztig left and right cells of $S_n$ and the
tableaux arising from the Robinson-Schensted process one needs to be careful since this is
affected by how the elements of the (abstract) Coxeter group act
on the set $\{1,\ldots,n\}$, whether on the right or on the left.

At this point we recall briefly the generalizations of the notions of diagram and tableau,
commonly used in the basic theory, see~\cite{MPa15} for a more detailed description.
A \emph{diagram} $D$ is a non-empty finite subset of $\mathbb{Z}^2$.
We will assume  that $D$ has no empty rows or columns.
These are the principal diagrams of \cite{MPa15}.
We will also assume that both rows and columns of $D$ are indexed
consecutively from 1; a node in $D$ will be given coordinates $(a,b)$ where $a$ and $b$ are the indices respectively of the row and column which the node belongs to
(rows are indexed from top to bottom and columns from left to right).
The \emph{row-composition} $\lambda_D$ (respectively,
\emph{column-composition} $\mu_D$) of $D$ is defined by
setting $\lambda_{D,k}$ (respectively, $\mu_{D,k}$) to be the number of nodes on
the $k$-th row (respectively, column) of $D$.
If $\lambda$ and $\mu$ are compositions of $n$,
we will write $\mathcal{D}^{(\lambda,\mu)}$ for the set
of (principal) diagrams $D$ with $\lambda_D=\lambda$ and $\mu_D=\mu$.
We also define $\mathcal{D}^{(\lambda)}=\bigcup_{\mu\vDash n}\mathcal{D}^{(\lambda,\mu)}$.
If $\nu\vdash n$, the \emph{Young diagram} associated with $\nu$ is the unique element of $\mathcal D^{(\nu,\nu')}$.
A \emph{special diagram} is a diagram obtained from a Young diagram
by permuting the rows and columns
(see~\cite[Proposition~3.1]{MPa15} for a characterization of special diagrams).

We say that a diagram $D$ has \emph{size} $n$ if it consists of precisely $n$ nodes.
We also define the \emph{length of a column} of a diagram $D$ to be the number of nodes $D$ has on this column.
If $D$ has exactly $m$ columns, we set $\alpha_D=(\alpha_1,\ldots,\alpha_m)$ where $\alpha_i$ equals the length of column $i$ of $D$, for $1\LEQ i\LEQ m$.
We call the $m$-tuple $\alpha_D$ the \emph{tuple of column-lengths} of~$D$.

If $D$ is a diagram of size $n$,
a \emph{$D$-tableau} is a bijection
$t\Colon D\rightarrow \{1,\ldots,n\}$ and
we refer to $(i,j)t$, where $(i,j)\in D$, as the
$(i,j)$-\emph{entry} of~$t$.
The group $W$ acts on the set of $D$-tableaux in the obvious
way---if $w\in W$,
an entry $i$ is replaced by $iw$ and $tw$ denotes the tableau
resulting from the action of $w$ on the tableau $t$.
We denote by $t^{D}$ and  $t_{D}$ the two $D$-tableaux obtained by
filling the nodes of $D$ with $1,\ldots,n$ by rows and by columns,
respectively, and we write $w_{D}$ for the element of $W$ defined by
$t^{D}w_{D}=t_{D}$.

Now let $D$ be a diagram and let $t$ be a \emph{$D$-tableau}.
We say $t$ is \emph{row-standard} if it is increasing on rows.
Similarly, we say $t$ is \emph{column-standard} if it is increasing
on columns.
We say that $t$ is \emph{standard} if $(i',j')t\LEQ (i'',j'')t$
for any $(i',j'),(i'',j'')\in D$ with $i'\LEQ i''$ and $j'\LEQ j''$.
Clearly a standard $D$-tableau is row-standard and column-standard,
however the converse is not true, in general.

The following result will turn out to be useful in various arguments in Sections~\ref{sec:3a} and~\ref{sec5}.

\begin{result}[{\cite[Proposition 3.5]{MPa15}.
See also \cite[Section~2]{MPa21}.
Compare \cite[Lemma 1.5]{DJa86}}]
\label{res:3a}
Let $D$ be a diagram.
Then the mapping $u\mapsto t^{D}u$ is a bijection of the set
of prefixes of $w_{D}$ to the set of standard
$D$-tableaux.
\proofblock
\end{result}

Since $\mathfrak{X}_{J(\lambda_D)}=\{w\in S_n\colon t^Dw$ is row-standard$\}$, see~\cite[Lemma~1.1]{DJa86}, it follows that $w_D$ and all its prefixes belong to $\mathfrak{X}_{J(\lambda_D)}$.

In general, an element of $W$ has an expression of the form
$w_D$ for many different diagrams $D$ of size $n$.
If $\lambda\vDash n$ and $d\in \mathfrak{X}_{J(\lambda)}$, a way to locate suitable diagrams $D\in\mathcal{D}^{(\lambda)}$  with $d=w_D$ is given in~{\cite[Proposition~3.7]{MPa15}}.
The proof involves the construction of a very particular diagram $D=D(d,\lambda)\in\mathcal D^{(\lambda)}$ with $w_D=d$.
Moreover, in~\cite[Proposition~3.8]{MPa15} it is shown that among all diagrams $E\in\mathcal D^{(\lambda)}$ with $w_E=d$, diagram $D(d,\lambda)$ is the unique one with the minimum number of columns.

As in~\cite{MPa17}, for a composition $\lambda$ of $n$, we define the following subsets
of $\mathfrak{X}_{J(\lambda)}$ and $\mathcal{D}^{(\lambda)}$:
\begin{equation*}
\renewcommand{\arraystretch}{1.0}
\label{eqn:1a}
\begin{array}{rcl}
Z(\lambda) & = &
\{e\in\mathfrak{X}_{J(\lambda)}\colon w_{J(\lambda)}e\sim_Rw_{J(\lambda)}\},
\\
Z_s(\lambda) & = &
\{e\in Z(\lambda)\colon e=w_D
\mbox{ for some special diagram } D\in\mathcal{D}^{(\lambda)}\},
\\
Y(\lambda) & = &
\{x\in Z(\lambda)\colon x \mbox{ is not a prefix of any other }
y\in Z(\lambda)\},
\\
Y_s(\lambda) & = &
Y(\lambda)\cap Z_s(\lambda)=\{y\in Y(\lambda)\colon D(y,\lambda)\mbox{ is special}\},
\\
\mathcal{E}^{(\lambda)} & = &
\{D(y,\lambda)\colon y\in Y(\lambda)\}\ \mbox{ and }\
\mathcal{E}_s^{(\lambda)}=\{D\in\mathcal{E}^{(\lambda)}\colon D \mbox{ is special}\}.
\end{array}
\end{equation*}
\par
As we have already seen in Section~\ref{sec:2c}, $Z(\lambda)$ is a right ideal in $W$.
Moreover, the set
$w_{J(\lambda)}Z(\lambda)$ is the
right cell of $W$ containing $w_{J(\lambda)}$.
We denote this right cell by $\mathfrak{C}(\lambda)$.
The set $Y(\lambda)$ is the minimal determining set of the right ideal $Z(\lambda)$.
We also call $Y(\lambda)$ the \emph{rim} of the cell
$\mathfrak{C}(\lambda)$.
The map $y\mapsto D(y,\lambda)$ from $Y(\lambda)$ to $\mathcal{E}^{(\lambda)}$ is a bijection, so $Y(\lambda)=\{w_D\colon D\in\mathcal{E}^{(\lambda)}\}$.
Hence, in order to give an explicit description of $Y(\lambda)$ or $\mathfrak{C}(\lambda)$ it is enough to locate the diagrams in $\mathcal{E}^{(\lambda)}$.

The work in~\cite{Schensted1961} and~\cite{Greene1974} motivates the following definition.

\begin{definition}[{Compare with~\cite[Lemma~3.2]{MPa17}, the definition before Remark~3.3
in \cite{MPa17}, and~\cite[Definition~3.6]{MPa21}}]
\label{def:2.1a}
Let $D$ be a diagram of size $n$.

\vspace{-3ex}
\begin{enumerate}[(i)]
\item
A \emph{path} of \emph{length} $m$ in $D$ is a non-empty sequence of
nodes $((a_i,b_i))_{i=1}^m$ of $D$ such that $a_i<a_{i+1}$ and
$b_i\LEQ b_{i+1}$ for $i=1,\ldots,m-1$.
\vspace{-1.5ex}
\item
For $k\in\mathbb{N}$, a \emph{$k$-path} in $D$ is a sequence of $k$
mutually disjoint paths in $D$;
the paths in this sequence are the \emph{constituent paths} of the
$k$-path.
The \emph{length} of a $k$-path is the sum of the lengths of its
constituent paths; this is the total number of nodes in the $k$-path.
The \emph{type} of a $k$-path is the sequence of lengths of its paths
in non-strictly decreasing order---in particular, the type of a
$k$-path is a $k$-part partition.
The \emph{support} of a $k$-path $\Pi$, which we denote by $s(\Pi)$,
is the set of nodes occurring in its paths.
\vspace{-1.5ex}
\item
Let $\Pi$ be a $k$-path in  $D$ and let $k'\LEQ k$.
A \emph{$k'$-subpath} of $\Pi$ is a $k'$-path in $D$ whose constituent
paths are also constituent paths of $\Pi$.
\vspace{-1.5ex}
\item
Let $\Pi=(\pi_1,\dots,\pi_k)$ be a $k$-path in $D$
where $\pi_j=((a_{i,j},b_{i,j}))_{i=1}^{m_j}$, for
$1\LEQ j\LEQ k$.
$\Pi$ is said to be \emph{ordered} if whenever $j,j'\in\{1,\ldots,k\}$
with $j<j'$ and $(a_{i,j},b_{i,j})$ and $(a_{i',j'},b_{i',j'})$ are
nodes of $\pi_j$ and $\pi_{j'}$, respectively, with
$a_{i,j}\LEQ a_{i',j'}$, then $b_{i,j}<b_{i',j'}$.
\vspace{-1.5ex}
\item
A $k$-path and a $k'$-path in $D$ are said to be \emph{equivalent} to one
another if they have the same support. 
\vspace{-1.5ex}
\item
The diagram $D$ is said to be of \emph{subsequence type}
$\nu$, where $\nu=(\nu_1,\ldots,\nu_r)\vdash n$,
if the maximum length of a $k$-path in $D$ is
$\nu_1+\ldots+\nu_k$ whenever $1\LEQ k\LEQ r$.
We call $D$ \emph{admissible} if it is of subsequence
type $\lambda_D'$.
\end{enumerate}
\end{definition}

Below we collect some results in~\cite{MPa17} and~\cite{MPa21} about paths and admissible diagrams which will play some part in Sections~\ref{sec:3a} and~\ref{sec5}.

\begin{result}[{See \cite[Propositions 3.5 and~3.6~and~Corollary~3.7]{MPa17}}]
\label{res:8a}
Let $D$ be a diagram of size $n$ and let $\nu$ be a partition of $n$.

\vspace{-3ex}
\begin{enumerate}[(i)]
\item If $D$ is of subsequence type $\nu$ then
$\mu_D''\unlhd\nu\unlhd\lambda_D'$.

\item
We have $w_{J(\lambda_D)}w_D\sim_R w_{J(\lambda_D)}$
if, and only if, $D$ is admissible. 
\item
If $D=s(\Pi)$ for some $k$-path $\Pi$ in $D$ of type $\lambda'_D$, then $D$ is admissible.
In particular, if $D$ is a special diagram then
$D$ is admissible.
\end{enumerate}
\end{result}
Note, however, that for composition $\lambda$ it is not true in general that every admissible diagram $E\in\mathcal D^{(\lambda)}$ is the support of some $k$-path in $D$ of type $\lambda'$ ---
consider for example the diagram
{\tiny$\begin{array}{ll}\times&\times\\ \times& \\ &\times \\ \times&\times\end{array}$} in $\mathcal D^{((2,1,1,2))}$.

\begin{result}[{\cite[Theorem~3.13]{MPa21}}]
\label{thm:2.14a}
Let $k\GEQ1$ and suppose $\Pi$ is a $k$-path in a diagram $D$.
Then $\Pi$ is equivalent to an ordered $k$-path in $D$.
\end{result}

\begin{result}[{\cite[Corollary~3.16]{MPa21}}]
\label{cor:3.4a}
Let $\Pi=(\pi_1,\ldots,\pi_k)$ be an ordered $k$-path in a diagram
$D$, and let $(a_i',b_i')$, $1\LEQ i\LEQ l$, be $l$ distinct nodes of
$D$ which is not in $\Pi$.
If no path $\pi_j$, $1\LEQ j\LEQ k$, contains a pair of nodes of the
form $(a_{i,j,1},b_i')$, $(a_{i,j,2},b_i')$ with
$a_{i,j,1}<a_i'<a_{i,j,2}$ for any $i$ satisfying $1\LEQ i\LEQ l$,
then the paths $((a_i',b_i'))$ may be inserted into the sequence $\Pi$
to give an ordered $(k+l)$-path.
\end{result}

Finally for this section we recall two results from~\cite{MPa17} which relate the sets $\mathcal E^{(\lambda)}$ and $\mathcal E^{(\mu)}$ when composition $\mu$ is obtained from composition $\lambda$ in some particular ways.

The \emph{reverse} composition $\dot\lambda$ of a composition
$\lambda=(\lambda_1$, $\dots$, $\lambda_r)$ of $n$ is the composition
$(\lambda_r$, $\dots$, $\lambda_1)$ of $n$ obtained by reversing the order of
the entries.
For a  diagram $D\in\mathcal{D}^{(\lambda)}$,
the diagram $\dot D\in\mathcal{D}^{(\dot\lambda)}$ is the diagram
obtained by rotating $D$ through $180^{\circ}$.
If $D\in\mathcal{D}^{(\lambda,\mu)}$,
then $\dot D\in\mathcal{D}^{(\dot\lambda,\dot\mu)}$.

\begin{result}[{\cite[Proposition~3.9]{MPa17} and \cite[Remark~2.9]{MPa21}}]
\label{res:10a}
Let $\lambda,\mu\vDash n$.
The map $D\mapsto \dot D$ from $\mathcal D^{(\lambda,\mu)}$ to $\mathcal D^{(\dot\lambda,\dot\mu)}$ induces a bijection between the sets $\mathcal E^{(\lambda)}$ and $\mathcal E^{(\dot\lambda)}$.
\end{result}

Given a composition $\lambda=(\lambda_1,\ldots,\lambda_r)\vDash n$,
let $\lambda_{*}=(\lambda_1,\ldots,\lambda_r,1)\vDash n+1$.
In \cite[Section~4]{MPa17}, there is a well-defined mapping $\psi$ from
the set of admissible diagrams in $\mathcal{D}^{(\lambda)}$ to the set
of admissible diagrams in $\mathcal{D}^{(\lambda_*)}$.
For a given admissible diagram $D$ in $\mathcal{D}^{(\lambda)}$, the
diagram $D\psi$ is obtained by examining all diagrams constructed from
$D$ by appending an $(r+1)$-th row with a single node to $D$ and
selecting the diagram which is admissible and such that the column
of the new node is minimal.

\begin{result}[{\cite[Proposition~4.2]{MPa21}}]
\label{prop:3.2a}
Let $r\GEQ 2$, let $n\GEQ 2$ and
let $\lambda=(\lambda_1,\ldots,\lambda_r)\vDash n$
be an $r$-part composition with $\lambda_r=1$.
Let $\psi$ be the mapping described in \cite[Section~4]{MPa17}.
Then $\psi$ induces a bijection from $\mathcal E^{(\lambda)}$ to $\mathcal E^{(\lambda_*)}$.
\end{result}

\end{section}

\begin{section}
{Ordered $k$-path structure of admissible diagrams}
\label{sec:3a}

Most of the work in this section is concerned with the identification and investigation of certain key ordered $k$-paths which have as their support an admissible diagram $D\in\mathcal D^{(\lambda)}$.
Later on in the paper we show how these particular ordered $k$-paths lead to the determination of the set $\mathcal E^{(\lambda)}$ and thus to the determination of the rim of the Kazhdan-Lusztig cell $\mathfrak C(\lambda)$ (or, equivalently, to the determination of the minimal determining set of the $W$-graph ideal $Z(\lambda)$).
Some motivation in taking this approach is given by the proof of~\cite[Theorem~4.6]{MPa21} as one of its main ingredients is that, in the case $\lambda$ is a 3-part composition, any admissible diagram in $\mathcal D^{(\lambda)}$ is the support of an ordered $k$-path of type $\lambda'$.

Next, we focus on compositions $\lambda$ of the form $(\lambda_1,\lambda_2,\lambda_3,1^r)$.
We begin by fixing some notation.

\textbf{Hypothesis~(*):}
Let $s\GEQ t\GEQ u\GEQ 1$.
We say that the composition $\lambda$ satisfies Hypothesis~(*) if $\lambda=(\lambda_1,\lambda_2,\lambda_3,1)$ is a composition of $s+t+u+1$ where
$\tilde\lambda=(\lambda_1,\lambda_2,\lambda_3)$ is a permutation of
$(s,t,u)$.

We  continue with a study of the ordered $k$-paths in  an admissible diagram $D\in\mathcal{D}^{(\lambda)}$ where $\lambda$ is a composition satisfying
Hypothesis~(*).
Clearly, these diagrams have no paths of length greater than~4.
If $\Pi$ is $k$-path in $D$, we let $z_i(\Pi)$ be the number of
constituent paths in $\Pi$ of length $i$ for $1\LEQ i\LEQ 4$.
We make the following technical definition of two forms of ordered $k$-path
in the diagram $D$.
We justify this definition in Lemma~\ref{lem:3.12a}.
\begin{definition}
\label{def:3.9a}
Suppose that composition $\lambda$ satisfies Hypothesis $(*)$ and that $D\in\mathcal D^{(\lambda)}$ is an admissible diagram.
An ordered $s$-path $\Pi$ in $D$ of length $s+t+u+1$ which contains
a $t$-subpath of length $2t+u+1$ is said to be a \emph{form-A}
$s$-path if $z_1(\Pi)=s-t$, $z_2(\Pi)=t-u$, $z_3(\Pi)=u-1$, and
$z_4(\Pi)=1$ and a \emph{form-B} $s$-path if
$z_1(\Pi)=s-t$, $z_2(\Pi)=t-u-1$, $z_3(\Pi)=u+1$, and $z_4(\Pi)=0$.
\end{definition}

\begin{remark}\label{Remark:ordered_subpaths}
Under the hypothesis and notation of Definition~\ref{def:3.9a} we can make the following observations.\\
(i) If the constituents of the $t$-subpath in Definition~\ref{def:3.9a} are
listed in the same order as they appear in the $s$-path then the
$t$-path is also ordered.\\
(ii) Any form-A $s$-path in $D$ has type $\lambda'$.
In particular, if $D$ has a form-A $s$-path then $D$ is admissible (see Result~\ref{res:8a}(iii)).\\
(iii) If $D$ has a form-B $s$-path then $t>u$.
\end{remark}

It is possible for an admissible diagram $D\in\mathcal D^{(\lambda)}$, with $\lambda$ satisfying Hypothesis~(*), to have both
form-A and form-B $s$-paths as the following example shows.
\begin{example}
\label{ex:3.11a}

Let  $D\in\mathcal D^{((4,6,3,1))}$ be the diagram
$\begin{array}{*{8}{c}}
        & \times & \times &        &        & \times & \times &        \\
 \times &        & \times & \times &        & \times & \times & \times \\
        &        &        &        & \times & \times &        & \times \\
        &        &        &        &        & \times &        &        \\
\end{array}$.
The 6-paths $\Pi_1,\Pi_2\in D$ where
$\Pi_1=\{$
$\{$
$\{(2,1)\},$
$\{(1,2),$ $(2,3),$ $(3,5),$ $(4,6)\},$
$\{(1,3),$ $(2,4),$ $(3,6)\},$
$\{(2,6)\},$
$\{(1,6),$ $(2,7),$ $(3,8)\},$
$\{(1,7),$ $(2,8)\}$
$\}$
and
$\Pi_2=\{$
$\{$
$\{(2,1)\},$
$\{(1,2),$ $(2,3),$ $(4,6)\},$
$\{(1,3),$ $(2,4),$ $(3,5)\},$
$\{(1,6),$ $(2,6),$ $(3,6)\},$
$\{(1,7),$ $(2,7),$ $(3,8)\},$
$\{(2,8)\}$$\}$
are form-A and form-B, respectively.
\end{example}

In the next two lemmas, which will play an important part in the discussion that follows, we investigate the existence of form-A or form-B $s$-paths in admissible diagrams $D\in\mathcal D^{(\lambda)}$ with $\lambda$ satisfying Hypothesis~(*).

\begin{lemma}
\label{lem:3.12a}
Assume that composition $\lambda$ satisfies Hypothesis $(*)$ and that $D\in\mathcal D^{(\lambda)}$ is an admissible diagram.
Then $D$ is the support of an ordered $s$-path ${\Pi}$ of length $s+t+u+1$ which is either
a form-A $s$-path or a form-B $s$-path.
Moreover,
\\[1ex]
\begin{tabular}{rp{14.5cm}}
(i) & if there are paths in $\Pi$ of length 1 then $s>t$ and all nodes
occurring in such paths are in a row of $D$ of length $s$,
\\[1ex]
(ii) & if there are paths in $\Pi$ of length 2 then $t>u$ and all
nodes occurring in such paths are in the rows of $D$ of lengths $s$ and
$t$,
\\[1ex]
(iii) & all nodes occurring in paths of length 3 are in the first three
rows of $D$, except in the case that $\Pi$ is a form-B $s$-path, when
one of these paths has its nodes on the rows of lengths $s$ and $t$
and on the fourth row of $D$.
\end{tabular}
\end{lemma}

\begin{proof}
We assume the hypothesis in the statement of the lemma.
Choose distinct $i_1,i_2,i_3\in\{1,2,3\}$
so that $\lambda_{i_1}=s$, $\lambda_{i_2}=t$ and $\lambda_{i_3}=u$.
First note that an $s$-path in $D$ of length $s+t+u+1$ contains all
nodes of $D$.
We will construct an $s$-path $\Pi$ of this length in $D$ with the
stated properties.

Let $N=(4,l)$ be the fourth row node of $D$.
Since $D$ is an admissible diagram it has subsequence type
$\lambda'=4^13^{u-1}2^{t-u}1^{s-t}$.
Thus $D$ has a path of length 4 and every path in $D$ of length 4
contains $N$.
\par
Since $D$ is admissible it has $t$-paths of length $2t+u+1$ and no
$t$-paths of greater length (see Result~\ref{res:8a}(i)).
Let $\Pi'=(\pi_1',\ldots,\pi_t')$ be one of these $t$-paths.
Using Result~\ref{res:3a} we may take $\Pi'$ to be an ordered
$t$-path of length $2t+u+1$.
Using the notation introduced before Definition~\ref{def:3.9a},
let $z_i'=z_i(\Pi')$ for $1\LEQ i\LEQ 4$.
Counting paths and nodes in $\Pi'$,
\begin{equation}
\label{eqn:19a}
z_1'+z_2'+z_3'+z_4'=t,\hspace{0.25cm}
z_1'+2z_2'+3z_3'+4z_4'=2t+u+1.
\end{equation}
\\[-4.5ex]
So,
\\[-4ex]
\begin{equation}
\label{eqn:20a}
z_2'+2z_3'+3z_4'=t+u+1,\hspace{0.25cm}
3z_1'+2z_2'+z_3'=2t-u-1.
\end{equation}
\\[-4ex]
Hence,
\\[-3.5ex]
\begin{equation}
\label{eqn:21a}
z_2'+z_3'+z_4'\LEQ t,\hspace{0.25cm}
z_3'+2z_4'\GEQ u+1.
\end{equation}
Since every path of length 4 in $D$ contains $N$, $z_4'\LEQ1$.
Below we will consider the cases $z_4'=0$ and $z_4'=1$ separately.

As $D$ contains no $t$-paths with more than $2t+u+1$ nodes, none of
the nodes of $D$ which are not nodes of $\Pi'$ can be inserted into
a path of $\Pi'$ to form a larger path.
Hence, by Result~\ref{cor:3.4a}, $\Pi'$ may be extended to an
ordered $s$-path $\Pi$ by the appropriate insertion of the $s-t$
paths of length 1 determined by the remaining nodes.

Case $z_4'=0$:  Then $t-u-1=2z_1'+z_2'\GEQ0$ from (2) and (3), so the choice of $i_3$ is unique.
Moreover, from (4) we get $z_3'\GEQ u+1$.
Since any path in $\Pi$ of length 3, avoiding row $i_3$ must contain $N$, the $u$ nodes of row $i_3$ must all lie in paths of length 3, thus giving us $u$ of the paths of length 3.
Hence, there is exactly one additional path of length 3, its nodes being on rows $i_1$, $i_2$ and 4  (all the remaining paths of length 3  necessarily have their nodes on rows $i_1$, $i_2$ and $i_3$).
In particular, we have $z_3'=u+1$.
From~(3) we now get $z_2'=t-u-1$, hence $z_2'=2z_1'+z_2'$.
It follows that $z_1'=0$.
Since all nodes on rows $i_3$ and 4 are on paths of length 3, the paths of length 2 only involve nodes on rows $i_1$ and $i_2$.
The $u+1$ paths of length 3 contain $u+1$ nodes on
row $i_2$.
Hence the remaining $t-u-1$ nodes on row $i_2$ are on the
$t-u-1$ paths of length 2.
Moreover, the $s-t$ nodes of $D$ which are not nodes of $\Pi'$
are all on row $i_1$.
If it is possible to choose $i_1$ in more than one way then $s=t$
and $z_1(\Pi)=0$.
So an apparent ambiguity arises concerning the rows of $D$ containing
the nodes of paths in $\Pi$ of length 1 only if such paths do not
exist.
\par
Case $z_4'=1$:
The path in $\Pi'$ of length 4 contains one node on row
$i_3$ and each path in $\Pi'$ of length 3 contains a node of row
$i_3$.
Hence $z_3'\LEQ u-1$.
From \eqref{eqn:21a} we get $z_3'=u-1$, and from \eqref{eqn:20a}
we get $z_2'=t-u$ and $z_1'=0$.
Since all nodes on rows $i_3$ and 4 are on paths of lengths
3 and 4, the paths of length 2 involve only nodes on rows $i_1$ and
$i_2$.
If it is possible to choose $i_3$ in more than one way, then $t=u$ and $z_2'=0$.
So an apparent ambiguity arises concerning the rows of $D$ containing the nodes of paths in $\pi$ of length 2 only if such paths do not exist.
The path of length 4 and the $u-1$ paths of length 3 contain $u$ nodes
on row $i_2$.
Hence the remaining $t-u$ nodes on row $i_2$ are on the $t-u$ paths of
length 2.
Thus the $s-t$ nodes of $D$ which are not nodes of $\Pi'$
are all on row $i_1$.
If it is possible to choose $i_1$ in more than one way then $s=t$
and $z_1(\Pi)=0$.
So an apparent ambiguity arises concerning the rows of $D$ containing
the nodes of paths in $\Pi$ of length 1  only if such paths
do not exist.

Since $\Pi$ is ordered and either
$z_1(\Pi)=s-t$,
$z_2(\Pi)=t-u$,
$z_3(\Pi)=u-1$, and
$z_4(\Pi)=1$
or
$z_1(\Pi)=s-t$,
$z_2(\Pi)=t-u-1$,
$z_3(\Pi)=u+1$, and
$z_4(\Pi)=0$,
$\Pi$ is either a form-A $s$-path or a form-B $s$-path.
\end{proof}

\begin{remark}
Keeping the hypothesis of Lemma~\ref{lem:3.12a}, it follows from the proof of Lemma~\ref{lem:3.12a} that if $\Gamma$ is an ordered $s$-path in $D$ of length $s+t+u+1$ which is either a form-A $s$-path or a form-B $s$-path, then the distribution of the nodes of the paths in $\Gamma$ in the rows of $D$ are as set out in Lemma~\ref{lem:3.12a}(i), (ii), (iii).
This is because $\Gamma$ necessarily contains an ordered $t$-subpath of length $2t+u+1$ (see also Remark~\ref{Remark:ordered_subpaths}(i)).
\end{remark}

\begin{lemma}\label{lemma:lambda1=st}
Assume that composition $\lambda$ satisfies Hypothesis $(*)$ and that $D\in\mathcal D^{(\lambda)}$ is an admissible diagram.
If $\lambda_1=s$ or $t$, then $D$ is the support of a form-A $s$-path.
\end{lemma}

\begin{proof}
We assume the hypothesis in the statement of the lemma and suppose further that $D$ has no form-A $s$-paths.
By Lemma~\ref{lem:3.12a}, $D$ has a form-B $s$-path $\Pi$.
In particular, $t>u$ since $z_2(\Pi)=t-u-1\GEQ0$.
Moreover, the nodes of the paths in $\Pi$ are distributed as set out in the statement again of Lemma~\ref{lem:3.12a}.
Write $\Pi=(\pi_1,\ldots,\pi_s)$.
Also since $D$ is admissible it has a path $\pi$ of length 4.
Suppose $s(\pi)=\{(1,l_1),(2,l_2),(3,l_3),(4,l)\}$ where $N=(4,l)$ is the
unique node of $D$ on the fourth row; among all such paths of length
4 we will choose $\pi$ to be the path which first minimizes $l_1$,
then minimizes $l_2$ and finally minimizes $l_3$.
Let $\pi_j$ be the path in $\Pi$ which contains $N$.
Then $\pi_j$ has length 3 and its remaining nodes are on
rows of lengths $s$ and $t$ (recall $t>u$).
Since $z_3(\Pi)=u+1\GEQ 2$ and the remaining paths in $\Pi$ of length 3
have their nodes on the first three rows there is a path
$\pi_{j'}$ in $\Pi$ of length 3 with $s(\pi_{j'})=\{(1,l_1'),(2,l_2'),(3,l_3')\}$.
If $l_3'\LEQ l$, we would get a form-A $s$-path in $D$ by replacing the
paths $\pi_{j'}$ and $\pi_j$ in $\Pi$ by the paths with support $s(\pi_{j'})\cup\{N\}$ and
$s(\pi_j)-\{N\}$, respectively.
Since this is not so, every path in $\Pi$ of length 3 ending on
row 3 ends in a column strictly to the right of $N$.

If $u=\lambda_3$, the node $(3,l_3)$ is on a path in $\Pi$ of length 3 by
Lemma~\ref{lem:3.12a}(iii).
Since $l_3\LEQ l$, this is excluded by the previous paragraph.
Hence $u\ne\lambda_3$.

Suppose now that $u=\lambda_2$.
Then the node $(2,l_2)$ is on a path $\pi_{j''}$ in $\Pi$ of
length 3 by Lemma~\ref{lem:3.12a} (iii) and $j''\neq j$.
Write $\pi_{j''}=\{(1,l_1''),(2,l_2),(3,l_3'')\}$
and $\pi_j=\{(1,l_1^{*}),(3,l_3^{*}),(4,l)\}$.
As above $l<l_3''$ and so $j<j''$, $l_1^{*}<l_1''\,(\LEQ l_2)$ and $l_3^{*}<l_3''$.
Since $\{(1,l_1^{*}),(2,l_2),(3,l_3),(4,l)\}$ is the support of a path of length 4
in $D$, we have that $l_1\LEQ l_1^{*}$ by the minimal choice of $\pi$.
Hence $l_1\LEQ l_1^{*}\LEQ l_3^{*}\LEQ l<l_3''$ and
$l_1^{*}<l_1''\LEQ l_2\LEQ l_3\LEQ l<l_3''$.

Next we choose $j'$ minimal subject to  $j<j'\LEQ j''$ and $\pi_{j'}$ has length 3 and let $s(\pi_{j'})=\{(1,l'_1),\, (2,l_2'),\, (3,l_3')\}$.
Combining with our observations in the last paragraph we get that $l_1\LEQ l_1^*<l_1'\LEQ l_2'\LEQ l_2\LEQ l_3\LEQ l<l_3''$.
It follows that $\{(1,l_1),\ (2,l_2'),\, (3,l_3),\, (4,l)\}$ is the support of a path in $D$ of length 4.
The minimal choice of $\pi$ forces $l_2'\GEQ l_2$, and since $l_2'\LEQ l_2$ from above, we get that $l_2'=l_2$.
We conclude that $j'=j''$, so $\pi_{j'}=\pi_{j''}$.
In particular, the choice of $\pi_{j'}$ ensures that no path $\pi_i$ with $j<i<j''$ has length 3.
Hence, by Lemma~\ref{lem:3.12a} (i), (ii), no path $\pi_i$ with $j<i<j''$ has a second row node.

Suppose for a moment that $l_2\LEQ l_3^*$.
Then there are paths $\widehat{\pi}_{j}$ and $\widehat{\pi}_{j''}$ in $D$ with support $s(\pi_{j})\cup\{(2,l_2)\}$ and $s(\pi_{j''})-\{(2,l_2)\}$ respectively.
The assumption that $l_2\LEQ l_3^*$, together with the observation that no path $\pi_i$ with $j<i<j''$ has a second row node, ensure that the $s$-path
obtained by replacing the paths $\pi_{j}$ and $\pi_{j''}$ in $\Pi$ by
the paths $\widehat\pi_{j}$ and $\widehat\pi_{j''}$, respectively,
is a form-A $s$-path.
Since this is excluded, we have $l_3^{*}<l_2$.

Let $\pi_{j^0}$ be the path in $\Pi$ containing the node $(3,l_3)$.
Since $l_3^*<l_2$ and $l_2\LEQ l_3$, we get $l_3^*<l_3\LEQ l<l_3''$ (by combining with certain inequalities obtained above).
It follows that $j<j^0<j''$.
Hence, from the discussion in the last-but-one paragraph, $\pi_{j^0}$ has length at most 2 and does not contain any node in the second row.

Suppose first that $\pi_{j^0}$ has length 2, so $s(\pi_{j^0})=\{(1,l_1^0),(3,l_3)\}$ for some $l_1^0$ satisfying $l_1^*<l_1^0<l_1''$.
Recalling that $l_1''\LEQ l_2\LEQ l_3$, we see that the $s$-path obtained from $\Pi$ by replacing \\
(i) $\pi_j$ by $\tilde \pi_j$ where $s(\tilde\pi_j)=s(\pi_j)-\{(4,l)\}$, \\
(ii) $\pi_{j^0}$ by $\tilde\pi_{j^0}$ where $s(\tilde\pi_{j^0})=s(\pi_{j^0})\cup\{(2,l_2),\, (4,l)\}$, and \\
(iii) $\pi_{j''}$ by $\tilde\pi_{j''}$ where $s(\tilde\pi_{j''})=s(\pi_{j''})-\{(2,l_2)\}$, \\
is an (ordered) form-A $s$-path in $D$, a contradiction.

It follows that $\pi_{j^0}$ has length 1, so $s(\pi_{j^0})=\{(3,l_3)\}$.
Let $\Psi$ be the ordered $t$-subpath of $\Pi$ of length $2t+u+1$ consisting precisely of the paths of length $>1$ in $\Pi$ (keeping the order these paths have in $\Pi$).
Also let $\hat\Psi$ be the $t$-path of length $2t+u-1$ in $D$ obtained from $\Psi$ by replacing\\
(i) $\pi_j$ by $\check \pi_j$ where $s(\check\pi_j)=\{(1,l_1^*),\, (3,l_3)\}$, and\\
(ii) $\pi_{j''}$ by $\check\pi_{j''}$ where $s(\check\pi_{j''})=\{(1,l_1''),\, (3,l_3'')\}$.

In particular, $s(\hat\Psi)\cup\{(3,l_3^*),\,(4,l),\,(2,l_2)\}=s(\Psi)\cup\{(3,l_3)\}$.
Let $\Gamma$ be the $(r+2)$-subpath of $\hat\Psi$ consisting of $\check\pi_j$, $\check\pi_{j''}$ and all the $r$ paths (with $r\GEQ0$) of $\Psi$ of length 2 which lie strictly between $\pi_j$ and $\pi_{j''}$ in the ordering of $\Psi$.
Clearly $\Gamma$ has length $2r+4$.
By Result~\ref{res:3a} $\Gamma$ is equivalent to an ordered $(r+2)$-path $\Gamma^*$ in $D$ (of length $2r+4$).
Since $s(\Gamma^*)\,(=s(\Gamma))$ does not contain any second row nodes, the maximum length of a path in $\Gamma^*$ is 2.
Hence $\Gamma^*$ consists of precisely $r+2$ paths each of length 2.
Let $\pi^*_{j^0}$ be the path in $\Gamma^*$ containing $(3,l_3)$, so $s(\pi^*_{j^0})=\{(1,\bar l_1), (3,l_3)\}$ for some $\bar l_1$ with $\bar l_1\LEQ l_1''\,(\LEQ l_2\LEQ l_3\LEQ l)$.
Also let $\pi^*$ be the path in $D$ of length 4 with $s(\pi^*)=s(\pi^*_{j^0})\cup\{(2,l_2),\, (4,l)\}$.

Next, we construct the $t$-path $\hat\Gamma$ in $D$ with $s(\hat\Gamma)=(s(\Psi)\cup\{(3,l_3)\})-\{(3,l_3^*)\}$ as follows.\\
(i) Keeping the order the paths appear in $\Psi$, include all paths beginning from the first one up to and including the path immediately before $\pi_j$ (but not including $\pi_j$).
\\
(ii) Then include all the paths in $\Gamma^*$, in the order they appear in $\Gamma^*$, but with $\pi^*_{j^0}$ replaced by $\pi^*$.
\\
(iii) Finally, include all the paths in $\Psi$ appearing strictly after $\pi_{j''}$ keeping the order these paths have in $\Psi$.

By its construction, $\hat\Gamma$ is an ordered $t$-path in $D$ of length $2t+u+1$.
Using Result~\ref{cor:3.4a}, we can construct an ordered $s$-path $\hat\Pi$ of length $s+t+u+1$ in $D$ by inserting in the sequence $\hat\Gamma$ the $s-t$ paths of length 1, each having support a node of $D-s(\hat\Gamma)$.
To justify this, observe that if the node $(a,b)$ belongs to $D-s(\hat\Gamma)$, then the existence of nodes $(a_1,b)$ and $(a_2,b)$ in $\hat\Gamma$ with $a_1<a<a_2$ would imply the existence of a $t$-path in $D$ of length $2t+u+2$ which is not possible by Result~\ref{res:8a}(i).
Clearly $\hat\Pi$ is a form-A $s$-path in $D$.
Hence $u\ne\lambda_2$.

Summing up, we have shown that the assumption  that $D$ has no form-A $s$-paths implies that $t>u$ and $u=\lambda_1$.
The required result now follows easily.
\end{proof}

In the following remark we recall a definition and some results in~\cite{MPa21} which will turn out to be useful in the the discussion that follows.

\begin{remark}\label{rem_4.13}
Let $\lambda$ be a composition of $n$ and let $D\in\mathcal D^{(\lambda)}$.
Also assume that $D=s(\Pi)$ for some $k$-path $\Pi=(\pi_1,\ldots,\pi_k)$ in $D$.
As in~\cite[Definition~3.8]{MPa21} we denote by $D(\Pi)$ the diagram in $\mathcal D^{(\lambda)}$ constructed from $D$ by replacing each node of $\pi_j$ by a node on the same row but in column $j$, for $j=1,\ldots,k$.
Then

\vspace{-3ex}
\begin{enumerate}[(i)]
\item
If $\Pi$ is ordered, we have from~\cite[Lemma 3.9]{MPa21} that $t^{D(\Pi)}w_D$ is a standard $D(\Pi)$-tableau.

\item
In the special case $\lambda$ satisfies Hypothesis (*) and $\Pi$ is a form-A $s$-path in $D$ (then $\Pi$ is ordered and has type $\lambda'$), we have that $D(\Pi)$ is a special diagram in $\mathcal D^{(\lambda)}$.
Moreover, $w_D$ is a prefix of $w_{D(\Pi)}$ since from item (i) of this remark, $t^{D(\Pi)}w_D$ is a standard $D(\Pi)$-tableau (see Result~\ref{res:3a}).
\end{enumerate}
\end{remark}

\begin{example}
\label{ex:3.8a}
Suppose the composition  $\lambda$ satisfies Hypothesis~(*).
Suppose further that $\lambda_1=s$ or $t$.
By Lemma~\ref{lemma:lambda1=st} we know that any admissible diagram $D\in\mathcal D^{(\lambda)}$ has a form-A $s$-path, hence by Remark~\ref{rem_4.13}(ii) we know that $w_D$ is a prefix of $w_E$ for some special diagram $E\in\mathcal D^{(\lambda)}$.
Given now an admissible diagram $D\in\mathcal D^{(\lambda)}$, below we consider some particular examples of special diagrams $E\in\mathcal D^{(\lambda)}$ which could serve this purpose.
\par
(i): 
If $\tilde\lambda=(s,u,t)$ and $C\subseteq\{1,\ldots,t\}$ with $|C|=u$
and $v=\min(C)$, let $F_{C}=\{(1,i)\colon 1\LEQ i\LEQ s\}\cup
\{(2,i)\colon i\in C\}\cup\{(3,i)\colon 1\LEQ i\LEQ t\}\cup\{(4,v)\}$.
Clearly, the list of lengths of columns of $F_{C}$ is a
rearrangement of $\lambda'$.
It follows that $F_{C}$ is a special and hence admissible diagram in $\mathcal{D}^{(\lambda)}$.
\par
If $\tilde\lambda=(8,3,5)$ and $C=\{2,3,4\}$, then $F_{C}$ is the diagram
$\setlength{\arraycolsep}{3pt}
\begin{array}{*{8}{c}}
\times&\times&\times&\times&\times&\times&\times&\times\\
      &\times&\times&\times&      &      &      &      \\
\times&\times&\times&\times&\times&      &      &      \\
      &     \times &&      &      &      &      &      \\

\end{array}$

Let $C'\subseteq\{1,\ldots, t\}$ with $|C'|=u$ and suppose that $w_{F_{C'}}$ is a prefix of $w_{F_C}$.
Then the $F_C$-tableau $t^{F_C}w_{F_{C'}}$ is standard (and can be constructed by moving the entries of the $F_{C'}$-tableau $t_{F_{C'}}$ along the rows keeping their order, to the nodes of $F_C$.
Observe that row 1 (resp., row 3) of $F_C$ coincides with row 1 (resp., row 3) of $F_{C'}$ from the way these diagrams are constructed.
Moreover, in order to preserve standardness, we see that the nodes in row 2 of $F_C$ are in exactly the same positions as the nodes in row 2 of $F_{C'}$.
This forces $C=C'$.
We conclude that $F_C=F_{C'}$.

(ii): 
If $\tilde\lambda=(t,s,u)$, $C\subseteq\{1,\ldots,s-t+u\}$ with
$|C|=u$ and $v=\min (C)$, then
$G_{C}=\{(1,i)\colon i\in (C\cup\{s-t+u+1,\ldots,s\})\}\cup
\{(2,i)\colon 1\LEQ i\LEQ s\}\cup\{(3,i)\colon i\in C\}\cup\{(4,v)\}$
is a special and hence admissible diagram in $\mathcal{D}^{(\lambda)}$.
\par
If $\tilde\lambda=(5,8,3)$ and $C=\{2,4,5\}$, then $G_{C}$ is
the diagram
$
\setlength{\arraycolsep}{3pt}
\begin{array}{*{8}{c}}
      &\times&      &\times&\times&      &\times&\times\\
\times&\times&\times&\times&\times&\times&\times&\times\\
      &\times&      &\times&\times&      &      &      \\
      & \times     &      & &      &      &      &      \\
\end{array}$

Suppose now that $C'\subseteq\{1,\ldots,s-t+u\}$ with $|C'|=u$ and that $t^{G_C} w_{G_{C'}}$ is a standard $G_C$-tableau.
From the way diagrams $G_C$ and $G_{C'}$ are defined, we see that their second rows coincide.
Moreover, the last $t-u$ nodes in row 1 of these diagrams are in exactly the same positions.
In order to preserve standardness, the nodes in row 2 of this of these diagrams must also occupy the same positions.
Hence $C=C'$ and this forces $G_C=G_{C'}$.

(iii): 
If $\tilde\lambda=(t,u,s)$, 
let $C=\{v\}\cup\tilde C$   with $|\tilde C|=u-1$, $\tilde C\subseteq\{s-t+2,\ldots,s\}$ and
$v\in\{1,\ldots,\min(\tilde C)-1\}$.
(In particular, $v=\min(C)$ and $|C|=u$.)
Let $\tilde v=v$ or $s-t+1$ according as $v<s-t+1$ or not.
Then
$H_{C}=\{(1,i)\colon i=\tilde v\mbox{ or } s-t+2\LEQ i\LEQ s\}
\cup\{(2,i)\colon i\in C\}\cup\{(3,i)\colon 1\LEQ i\LEQ s\}\cup
\{(4,v)\}$ is a special and hence admissible diagram in $\mathcal{D}^{(\lambda)}$.

If $\tilde\lambda=(5,3,8)$ and $C=\{3,6,8\}$, then $H_{C}$ is
the diagram
$\setlength{\arraycolsep}{3pt}
\begin{array}{*{8}{c}}
      &      &\times&      &\times&\times&\times&\times\\
      &      &\times&      &      &\times&      &\times\\
\times&\times&\times&\times&\times&\times&\times&\times\\
      &      &\times&      &      &      &      &      \\
\end{array}$

Suppose now that $t^{H_C}w_{H_{C'}}$ is a standard $H_C$-tableau for permitted choices of $C$ and $C'$ above.
From the construction of $H_C$ and $H_{C'}$, the nodes in their third rows and moreover, the last $t-1$ nodes in their first rows, occupy exactly the same positions.
To preserve standardness, the last $u-1$ nodes in row 2 of the two diagrams must also be in exactly the same positions.
For the same reason, the node in row 4 (resp., first node in row 2) of $H_C$ is in exactly the same position as the node in row 4 (resp., first node in row 2) of $H_{C'}$.
Finally, from the way these diagrams are defined, their first nodes in row 1 are also forced to be in exactly the same position, proving that $H_C=H_{C'}$.
\end{example}

Next, we introduce some more notation.

\begin{definition}\label{def_Hypoth_dag}
Suppose composition $\lambda$ satisfies Hypothesis~(*) with the additional constraints $t>u$ and $\lambda_1=u$.
We then say that the diagram $D\in\mathcal D^{(\lambda)}$ satisfies Hypothesis~($\dag$) if in the associated  tuple $\alpha_D=(\alpha_1,\ldots,\alpha_m)$ of column-lengths of $D$ there is a single 4, exactly $(u-1)$ 3's and, in addition, the 4 occurs before all the 3's.
We also define the \emph{determining tuple} $\hat\alpha_D$ of $D$ by $\hat\alpha_D=(\hat\alpha_1,\ldots,\hat\alpha_m)$, where $\hat\alpha_j=\alpha_j$ if $\alpha_j\in\{2,3,4\}$, and $\hat\alpha_j=1$ or $\bar1$ according as the single node in a column $j$ of length 1 in $D$ is on row 2 or 3.
[Clearly the tuple $\hat\alpha_D$ determines a diagram satisfying Hypothesis ($\dag$) uniquely, since all columns in $D$ having length 2 (resp., length 3) necessarily have their nodes on rows 2 and 3 (resp., on rows 1, 2 and 3).]
\end{definition}

\begin{remark}\label{rem:operations}
Let $E\in\mathcal D^{(\lambda)}$ satisfy Hypothesis~($\dag$) (with composition $\lambda$ as in Definition~\ref{def_Hypoth_dag}) and let $\hat\alpha_E=(\hat\alpha_1,\ldots,\hat\alpha_m)$ be the determining tuple for $E$.
Suppose now that diagram $E'$ has been obtained from $E$ via any one of the operations (C1)--(C5) below.

Operations (C1)--(C4):
If $(\hat\alpha_j,\hat\alpha_{j+1})=(1,2)$ or $(3,2)$ or $(2,\bar1)$ or $(3,\bar1)$ for some $j\GEQ1$, diagram $E'$ is obtained from $E$ by interchanging the $j$-th and $(j+1)$-th columns.
Suppose for convenience that the first $j-1$ columns of $E$ contain exactly $\omega -1$ nodes (where $\omega\GEQ1$).
Then, from the way they are constructed, the tableaux $t_{E}$ and $t_{E'}$ differ only on these two columns, which respectively take
the form \\
{\setlength{\arraycolsep}{2pt}
\setlength{\fboxsep}{0pt}
$\begin{array}{|cc|}
\hline
        &  \\
\omega  & \omega\!+\!1 \\
        & \omega\!+\!2 \\
        &  \\
\hline
\end{array}$
and
$\begin{array}{|cc|}
\hline
             &  \\
\omega       &  \omega\!+\!2 \\
\omega\!+\!1 &  \\
             &  \\
\hline
\end{array}$ \ after operation (C1), {\setlength{\arraycolsep}{2pt}
\setlength{\fboxsep}{0pt}
$\begin{array}{|cc|}
\hline
\omega        &  \\
\omega\!+\!1  &  \omega\!+\!3 \\
\omega\!+\!2  &  \omega\!+\!4 \\
              &  \\
\hline
\end{array}$
and
$\begin{array}{|cc|}
\hline
             & \omega\!+\!2 \\
\omega       & \omega\!+\!3 \\
\omega\!+\!1 & \omega\!+\!4 \\
             &  \\
\hline
\end{array}$} \ after operation~(C2),
{\setlength{\arraycolsep}{2pt}
\setlength{\fboxsep}{0pt}
$\begin{array}{|cc|}
\hline
        &  \\
\omega  &  \\
\omega\!+\!1 & \omega\!+\!2 \\
        &  \\
\hline
\end{array}$
and
$\begin{array}{|cc|}
\hline
        &  \\
        & \omega\!+\!1 \\
\omega  & \omega\!+\!2 \\
        &  \\
\hline
\end{array}$} \ after operation (C3),
}
{\setlength{\arraycolsep}{2pt}
\setlength{\fboxsep}{0pt}
$\begin{array}{|cc|}
\hline
\omega        &  \\
\omega\!+\!1  &  \\
\omega\!+\!2        & \omega\!+\!3 \\
        &  \\
\hline
\end{array}$
and
$\begin{array}{|cc|}
\hline
        & \omega\!+\!1 \\
        & \omega\!+\!2 \\
\omega  & \omega\!+\!3 \\
        &  \\
\hline
\end{array}$} after operation (C4).
Since $t^{E'}w_{E}$ is standard, $w_{E}$ is a prefix of $w_{E'}$.

Operation (C5): If $\hat\alpha_j=2$ for some $j$, diagram $E'$ is obtained from $E$ after replacing the $j$-th column of $E$ by two adjacent columns each having a single node; the single node of the first one (resp., second one) being on row 3 (resp., row 2).
The difference in $t_E$ and $t_{E'}$ can be described by
{\setlength{\arraycolsep}{2pt}
\setlength{\fboxsep}{0pt}
$\begin{array}{|c|}
\hline
         \\
\omega   \\
\omega\!+\!1 \\
         \\
\hline
\end{array}$
and
$\begin{array}{|cc|}
\hline
       &  \\
       & \omega\!+\!1 \\
\omega &  \\
       &  \\
\hline
\end{array}$, respectively.
}
In particular, $E'$ has $m+1$ columns.

Clearly, in all of the above cases $E'\in\mathcal D^{(\lambda)}$ and $E'$ satisfies Hypothesis~($\dag$).
Moreover,  $t^{E'}w_{E}$ is a standard $E'$-tableau, so $w_{E}$ is a prefix of $w_{E'}$.
\end{remark}

\begin{example}
\label{ex:3.10a}
Let composition $\lambda$ satisfy Hypothesis~(*).
Suppose further that $u<t$ and $\lambda_1=u$.
In this example we introduce certain types of admissible diagrams $D\in\mathcal D^{(\lambda)}$ satisfying Hypothesis~($\dag$).
In particular, such diagrams $D$ cannot be transformed using operations (C1)--(C4) to an admissible diagram $K$ such that $t^Kw_D$ is a standard $K$-tableau.
Moreover, in all cases, they are the support of a form-B $s$-path.

(i): 
If $\tilde\lambda=(u,s,t)$, let $\varepsilon$, $\eta$, $\theta$,
$\zeta$ and $\psi$ be non-negative integers satisfying
$s=\varepsilon+\eta+1+\zeta+\psi$,
$t=\varepsilon+\theta+\zeta+u$, $\psi\GEQ u-1$, and $\eta\GEQ\theta$,
let $\mathcal{C}$ be a $(u-1)$-subset of
$\{s+\theta-\psi+1,\ldots,s+\theta\}$, and let
$\mathcal{S}=(\varepsilon,\eta,\theta,\zeta,\psi,\mathcal{C})$.
Define $M^{(\mathcal{S})}$ to be the diagram with nodes
\\[1ex]
$\begin{array}{lll}
(1,i) & \colon & i=\varepsilon+\eta+1\mbox{ or }i\in\mathcal{C}, \\[1ex]
(2,i) & \colon & 1\LEQ i\LEQ\varepsilon+\eta+1\mbox{ or }
\varepsilon+\eta+\theta+2\LEQ i\LEQ s+\theta, \\[1ex]
(3,i) & \colon &  1\LEQ i\LEQ\varepsilon\mbox{ or }
\varepsilon+\eta+1\LEQ i\LEQ\varepsilon+\eta+\theta+\zeta+1\mbox{ or }
i\in\mathcal{C}, \\[1ex]
(4,i) & \colon & i=\varepsilon+\eta+1.
\end{array}$
\\[1ex]
Then $M^{(\mathcal{S})}\in\mathcal D^{(\lambda)}$, it has  $s+\theta$ columns and it is easy to check that it satisfies Hypothesis~($\dag$).
In the determining tuple of $M^{(\mathcal{S})}$ there is a single 4  and the tuple occurring before the 4 consists of an $\varepsilon$-tuple of 2's followed by an $\eta$-tuple of 1's.
Following the 4 there is a $\theta$-tuple of $\bar1$'s followed by a $\zeta$-tuple of 2's followed by a $\psi$-tuple of 1's and 3's containing exactly $(u-1)$ 3's.

Schematically, $M^{(\mathcal{S})}$ takes the form
\[
\setlength{\arraycolsep}{0.175em}
\newcommand{\X}{\times}
\begin{array}{*{10}{c}}
\overbrace{
\begin{array}[t]{*{4}{c}}
\mathstrut     &        &    \\
\mathstrut  \X & \cdots & \X \\
\mathstrut  \X & \cdots & \X \\
\mathstrut     &        &    \\
\end{array}
}^{\varepsilon}
&
\overbrace{
\begin{array}[t]{*{3}{c}}
\mathstrut     &        &    \\
\mathstrut  \X & \cdots & \X \\
\mathstrut     &        &    \\
\mathstrut     &        &    \\
\end{array}
}^{\eta}
&
\overbrace{
\begin{array}[t]{*{1}{c}}
\mathstrut  \X \\
\mathstrut  \X \\
\mathstrut  \X \\
\mathstrut  \X \\
\end{array}
}^{1}
&
\overbrace{
\begin{array}[t]{*{3}{c}}
\mathstrut     &        &    \\
\mathstrut     &        &    \\
\mathstrut  \X & \cdots & \X \\
\mathstrut     &        &    \\
\end{array}
}^{\theta}
&
\overbrace{
\begin{array}[t]{*{3}{c}}
\mathstrut     &        &    \\
\mathstrut  \X & \cdots & \X \\
\mathstrut  \X & \cdots & \X \\
\mathstrut     &        &    \\
\end{array}
}^{\zeta}
&
\overbrace{
\begin{array}[t]{*{7}{c}}
\mathstrut     &        & \X &        &    &        & \X \\
\mathstrut  \X & \cdots & \X & \cdots & \X & \cdots & \X \\
\mathstrut     &        & \X &        &    &        & \X \\
\mathstrut     &        &    &        &    &        &    \\
\end{array}
}^{\psi}
\end{array}\ .
\]

From the construction of $M^{(\mathcal S)}$ we see that $M^{(\mathcal S)}$ is special if, and only if, $\theta=0$.
It is immediate that $M^{(\mathcal S)}$ is special if $\theta=0$.
Conversely, if $M^{(\mathcal S)}$ is special and $s>t$, it is clear that $\theta=0$.
Finally, if $s=t$, the relations in line~2 of this example imply that $\eta=\theta$ and $\psi=u-1$, so the additional constraint that $M^{(\mathcal S)}$ is special, now implies that $\eta=\theta=0$.
In particular we have that $M^{(\mathcal S)}$ is admissible if $\theta=0$.

If $\eta\GEQ\theta\GEQ1$,
let $\Pi=(\pi_1,\ldots,\pi_s)$ and $\Pi'=(\pi_1',\ldots,\pi_s')$
be the $s$-paths defined by
\\[1ex]
$
\pi_j=
\left\{
\begin{array}{lll}
\{
(2,j),(3,j)
\}
&\mbox{ if }&
1\LEQ j\LEQ\varepsilon,
\\[1ex]
\{
(2,j),(3,j+\eta+1)
\}
&\mbox{ if }&
\varepsilon+1\LEQ j\LEQ\varepsilon+\theta,
\\[1ex]
\{
(2,j)
\}
&\mbox{ if }&
\varepsilon+\theta+1\LEQ j\LEQ\varepsilon+\eta,
\\[1ex]
\{
(1,j),(2,j),(3,j),(4,j)
\}
&\mbox{ if }&
j=\varepsilon+\eta+1,
\\[1ex]
\{
(2,j+\theta),(3,j+\theta)
\}
&\mbox{ if }&
\varepsilon+\eta+2\LEQ j\LEQ s-\psi, 
\\[1ex]
\{
(1,j+\theta),(2,j+\theta),(3,j+\theta)
\}
&\mbox{ if }&
s-\psi+1
\LEQ j\LEQ s
\mbox{ and }j+\theta\in\mathcal{C},
\\[1ex]
\{
(2,j+\theta)
\}
&\mbox{ if }&
s-\psi+1
\LEQ j\LEQ s
\mbox{ and }j+\theta\notin\mathcal{C}.
\end{array}
\right.
$
\\[1ex]
and
\\[1ex]
$
\pi_j'=
\left\{
\begin{array}{lll}
\{
(2,j),(3,j)
\}
&\mbox{ if }&
1\LEQ j\LEQ\varepsilon,
\\[1ex]
\{
(2,j),(3,j+\eta),(4,j+\eta)
\}
&\mbox{ if }&
j=\varepsilon+1,
\\[1ex]
\{
(2,j),(3,j+\eta)
\}
&\mbox{ if }&
\varepsilon+2\LEQ j\LEQ\varepsilon+\theta,
\\[1ex]
\{
(2,j)
\}
&\mbox{ if }&
\varepsilon+\theta+1\LEQ j\LEQ\varepsilon+\eta,
\\[1ex]
\{
(1,j),(2,j),(3,j+\theta)
\}
&\mbox{ if }&
j=\varepsilon+\eta+1,
\\[1ex]
\{
(2,j+\theta),(3,j+\theta)
\}
&\mbox{ if }&
\varepsilon+\eta+2\LEQ j\LEQ s-\psi, 
\\[1ex]
\{
(1,j+\theta),(2,j+\theta),(3,j+\theta)
\}
&\mbox{ if }&
s-\psi+1
\LEQ j\LEQ s
\mbox{ and }j+\theta\in\mathcal{C},
\\[1ex]
\{
(2,j+\theta)
\}
&\mbox{ if }&
s-\psi+1
\LEQ j\LEQ s
\mbox{ and }j+\theta\notin\mathcal{C}.
\end{array}
\right.
$

Then $\Pi$ has type $\lambda'$, so that $M^{\mathcal{(S)}}$ is an
admissible diagram (see Remark~\ref{Remark:ordered_subpaths}).
Also, $\Pi'$ is a form-B $s$-path.

We look at a specific case.
If $\tilde\lambda=(3,8,5)$ and
$\mathcal{S}=(1,3,1,0,3,\{7,8\})$
then\break
$
M^{(\mathcal{S})}=
\begin{array}{*{9}{c}}
        &        &        &        & \times &        & \times & \times &        \\
 \times & \times & \times & \times & \times &        & \times & \times & \times \\
 \times &        &        &        & \times & \times & \times & \times &        \\
        &        &        &        & \times &        &        &        &        \\
\end{array}$
\par
Let $\Pi=
\{
$ \{(2,1),$ $(3,1)\},
$ \{(2,2),$ $(3,6)\},
$ \{(2,3)\},
$ \{(2,4)\},
$ \{(1,5),$ $(2,5),$ $(3,5),$ $(4,5)\},
$ \{(1,7),$ $(2,7),$ $(3,7)\},
$ \{(1,8),$ $(2,8),$ $(3,8)\},
$ \{(2,9)\}
\}$
and
$\Pi'=
\{
$ \{(2,1),$ $(3,1)\},
$ \{(2,2),$ $(3,5),$ $(4,5)\},
$ \{(2,3)\},
$ \{(2,4)\},
$ \{(1,5),$ $(2,5),$ $(3,6)\},
$ \{(1,7),$ $(2,7),$ $(3,7)\},
$ \{(1,8),$ $(2,8),$ $(3,8)\},
$ \{(2,9)\}
\}$.
Then $\Pi$ and $\Pi'$ are $8$-paths in the diagram $M^{(\mathcal{S})}$ of
types $\lambda'=4^13^22^21^3$ and $3^42^11^3$, respectively.
Thus $M^{(\mathcal{S})}$ is admissible.
Moreover $\Pi'$ is a form-B $8$-path.

(ii): 
If $\tilde\lambda=(u,t,s)$, let
$\eta$, $\varepsilon$, $\theta$, $\varphi$, and $\zeta$ be
non-negative integers satisfying $s=\eta+\varepsilon+\varphi+\zeta+u$
and $t=\varepsilon+\theta+\zeta+u$, and $\varphi\GEQ\theta$.
Let $\mathcal{S}=(\eta,\varepsilon,\theta,\varphi,\zeta)$.
Define $N^{(\mathcal{S})}$ to be the diagram with nodes
\\[1ex]
$\begin{array}{lll}
(1,i)&\colon& i=\eta+\varepsilon+\theta+1\mbox{ or }
s+\theta-u+2\LEQ i\LEQ s+\theta,
\\[1ex]
(2,i)&\colon& \eta+1\LEQ i\LEQ\eta+\varepsilon+\theta+1\mbox{ or }
s+\theta-u-\zeta+2\LEQ i\LEQ s+\theta,
\\[1ex]
(3,i)&\colon& 1\LEQ i\LEQ\eta+\varepsilon\mbox{ or }
\eta+\varepsilon+\theta+1\LEQ i\LEQ s+\theta,
\\[1ex]
(4,i)&\colon& i=\eta+\varepsilon+\theta+1.
\end{array}
$
\\[1ex]
Then $N^{(\mathcal{S})}\in\mathcal D^{(\lambda)}$, it has $s+\theta$ columns and it is easy to check that it satisfies Hypothesis~($\dag$).
The determining tuple of $N^{(\mathcal{S})}$ consists, going from left to right, of an $\eta$-tuple of $\bar1$'s followed by an $\varepsilon$-tuple of 2's, a $\theta$-tuple of $1$'s, a single 4, a $\varphi$-tuple of $\bar1$'s, a $\zeta$-tuple of $2$'s and, finally, a $(u-1)$-tuple of $3$'s.

Schematically, $N^{(\mathcal{S})}$ takes the form
\[
\setlength{\arraycolsep}{0.175em}
\newcommand{\X}{\times}
\begin{array}{*{10}{c}}
\overbrace{
\begin{array}[t]{*{3}{c}}
\mathstrut     &        &    \\
\mathstrut     &        &    \\
\mathstrut  \X & \cdots & \X \\
\mathstrut     &        &    \\
\end{array}
}^{\eta}
&
\overbrace{
\begin{array}[t]{*{4}{c}}
\mathstrut     &        &    \\
\mathstrut  \X & \cdots & \X \\
\mathstrut  \X & \cdots & \X \\
\mathstrut     &        &    \\
\end{array}
}^{\varepsilon}
&
\overbrace{
\begin{array}[t]{*{3}{c}}
\mathstrut     &        &    \\
\mathstrut  \X & \cdots & \X \\
\mathstrut     &        &    \\
\mathstrut     &        &    \\
\end{array}
}^{\theta}
&
\overbrace{
\begin{array}[t]{*{1}{c}}
\mathstrut  \X \\
\mathstrut  \X \\
\mathstrut  \X \\
\mathstrut  \X \\
\end{array}
}^{1}
&
\overbrace{
\begin{array}[t]{*{3}{c}}
\mathstrut     &        &    \\
\mathstrut     &        &    \\
\mathstrut  \X & \cdots & \X \\
\mathstrut     &        &    \\
\end{array}
}^{\varphi}
&
\overbrace{
\begin{array}[t]{*{3}{c}}
\mathstrut     &        &    \\
\mathstrut  \X & \cdots & \X \\
\mathstrut  \X & \cdots & \X \\
\mathstrut     &        &    \\
\end{array}
}^{\zeta}
&
\overbrace{
\begin{array}[t]{*{3}{c}}
\mathstrut \X &        & \X \\
\mathstrut \X & \cdots & \X \\
\mathstrut \X &        & \X \\
\mathstrut    &        &    \\
\end{array}
}^{u-1}
\end{array}\ .
\]

It is easy to see that $N^{(\mathcal S)}$ is special if, and only if, $\theta=0$.
It follows that $N^{(\mathcal S)}$ is admissible if $\theta=0$.
Observe also that the relations $s$ and $t$ satisfy, force $\varphi=0$ and $\eta=0$ if $s=t$.
So in the case $s=t$ the shape of $M^{(\mathcal S)}$ in fact coincides with the shape of $N^{(\mathcal S)}$.

If $\varphi\GEQ\theta\GEQ1$,
let $\Pi=(\pi_1,\ldots,\pi_s)$ and $\Pi'=(\pi_1',\ldots,\pi_s')$  be
the $s$-paths defined by
\\[1ex]
\par
$
\pi_j=
\left\{
\begin{array}{lll}
\{
(3,j)
\}
&\mbox{ if }&
1\LEQ j\LEQ\eta,
\\[1ex]
\{
(2,j),(3,j)
\}
&\mbox{ if }&
\eta+1\LEQ j\LEQ\eta+\varepsilon,
\\[1ex]
\{
(2,j),(3,j+\theta+1)
\}
&\mbox{ if }&
\eta+\varepsilon+1\LEQ j\LEQ\eta+\varepsilon+\theta,
\\[1ex]
\{
(1,j),(2,j),(3,j),(4,j)
\}
&\mbox{ if }&
j=\eta+\varepsilon+\theta+1,
\\[1ex]
\{
(3,j+\theta)
\}
&\mbox{ if }&
\eta+\varepsilon+\theta+2\LEQ j\LEQ \eta+\varepsilon+\varphi+1, 
\\[1ex]
\{
(2,j+\theta),(3,j+\theta)
\}
&\mbox{ if }&
\eta+\varepsilon+\varphi+2\LEQ j\LEQ \eta+\varepsilon+\varphi+\zeta+1,
\\[1ex]
\{
(1,j+\theta),(2,j+\theta),(3,j+\theta)
\}
&\mbox{ if }&
\eta+\varepsilon+\varphi+\zeta+2\LEQ j\LEQ\eta+\varepsilon+\varphi+\zeta+u,
\end{array}
\right.
$
\\[1ex]
and
\\[1ex]
$
\pi_j'=
\left\{
\begin{array}{lll}
\{
(3,j)
\}
&\mbox{ if }&
1\LEQ j\LEQ\eta,
\\[1ex]
\{
(2,j),(3,j)
\}
&\mbox{ if }&
\eta+1\LEQ j\LEQ\eta+\varepsilon,
\\[1ex]
\{
(2,j),(3,j+\theta),(4,j+\theta)
\}
&\mbox{ if }&
j=\eta+\varepsilon+1,
\\[1ex]
\{
(2,j),(3,j+\theta)
\}
&\mbox{ if }&
\eta+\varepsilon+2\LEQ j\LEQ\eta+\varepsilon+\theta,
\\[1ex]
\{
(1,j),(2,j),(3,j+\theta)
\}
&\mbox{ if }&
j=\eta+\varepsilon+\theta+1,
\\[1ex]
\{
(3,j+\theta)
\}
&\mbox{ if }&
\eta+\varepsilon+\theta+2\LEQ j\LEQ \eta+\varepsilon+\varphi+1,
\\[1ex]
\{
(2,j+\theta),(3,j+\theta)
\}
&\mbox{ if }&
\eta+\varepsilon+\varphi+2\LEQ j\LEQ\eta+\varepsilon+\varphi+\zeta+1,
\\[1ex]
\{
(1,j+\theta),(2,j+\theta),(3,j+\theta)
\}
&\mbox{ if }&
\eta+\varepsilon+\varphi+\zeta+2
\LEQ j\LEQ \eta+\varepsilon+\varphi+\zeta+u.
\end{array}
\right.
$

Then $\Pi$ has type $\lambda'$, so that $N^{\mathcal{(S)}}$ is an
admissible diagram.
Also, $\Pi'$ is a form-B $s$-path.

We look at a specific case.
If $\tilde\lambda=(3,5,8)$ and $\mathcal{S}=(3,0,1,1,1,3)$ then\break
$
N^{(\mathcal{S})} =
\begin{array}{*{9}{c}}
        &        &        &        & \times &        &        & \times & \times \\
        &        &        & \times & \times &        & \times & \times & \times \\
 \times & \times & \times &        & \times & \times & \times & \times & \times \\
        &        &        &        & \times &        &        &        &        \\
\end{array}.
$
\par
Let
$\Pi=\{$
$\{(3,1)\},$
$\{(3,2)\},$
$\{(3,3)\},$
$\{(2,4),$ $(3,6)\},$
$\{(1,5),$ $(2,5),$ $(3,5),$ $(4,5)\},$
$\{(2,7),$ $(3,7)\},$
$\{(1,8),$ $(2,8),$ $(3,8)\},$
$\{(1,9),$ $(2,9),$ $(3,9)\}$
$\}$
and
$\Pi'=\{$
$\{(3,1)\},$
$\{(3,2)\},$
$\{(3,3)\},$
$\{(2,4),$ $(3,5),$ $(4,5)\},$
$\{(1,5),$ $(2,5),$ $(3,6)\},$
$\{(2,7),$ $(3,7)\},$
$\{(1,8),$ $(2,8),$ $(3,8)\},$
$\{(1,9),$ $(2,9),$ $(3,9)\}$
$\}$.
\\
Then $\Pi$ and $\Pi'$ are $8$-paths in the diagram $N^{(\mathcal{S})}$ of
types $\lambda'=4^13^22^21^3$ and $3^42^11^3$, respectively.
Thus $N^{(\mathcal{S})}$ is admissible.
Moreover $\Pi'$ is a form-B $8$-path.
\end{example}

In the course of the proof of Theorem~\ref{thm:3.13a} we will show that any diagram $D$ satisfying Hypothesis~($\dag$), and which also satisfies some additional constraints, can be transformed using operations (C1)--(C5) to a diagram which either equals $M^{(\mathcal S)}$ or $N^{(\mathcal S)}$ for a suitable tuple $S$.
The following two lemmas will also turn out to be useful.

\begin{lemma}\label{lemma:4.7}
Under the hypothesis and notation of Example~\ref{ex:3.10a}(i), let $K=M^{(\mathcal S)}$ and $K'=M^{(\mathcal S')}$ be the diagrams corresponding to $\mathcal S=(\varepsilon,\eta,\theta,\zeta,\psi,\mathcal{C})$ and $\mathcal S'=(\varepsilon',\eta',\theta',\zeta',\psi',\mathcal{C}')$, respectively, where $\eta\GEQ\theta$ and $\eta'\GEQ\theta'$.
Suppose that the further constraints $\zeta=0$ if $\eta>\theta$ and $\zeta'=0$ if $\eta'>\theta'$ are in force.
If $w_{K'}$ is a prefix of $w_K$, then $K=K'$.
\end{lemma}
\begin{proof}
We assume the hypothesis and we suppose that $w_{K'}$ is a prefix of $w_K$.
Then $t^Kw_{K'}$ is a standard $K$-tableau (see Result~\ref{res:3a}).
Observe that $t^Kw_{K'}$ is obtained from the $K'$-tableau $t_{K'}$ by moving the entries of each row of $t_{K'}$, keeping their order, to the nodes of $K$ on the same row.
This forces the entry on row 4 of $t_{K'}$ to move to the single node on row 4 of $K$ and, moreover the first entry (starting from the left) on row 1 of $t_{K'}$ to move to the first node on row 1 of $K$.

Thus, in order to preserve standardness, the column of length 4 in $t_{K'}$ must move to the column of length 4 in $K$.
Counting nodes on rows 3 and 2 of $K$ and $K'$ lying to the left of the column of length 4 in each of the two diagrams, we immediately get $\varepsilon=\varepsilon'$ and $\eta=\eta'$.
On the other hand, counting nodes on rows 2 and 3 lying to the right of the column of length 4 in each of the two diagrams, we get $\zeta+\psi=\zeta'+\psi'$ and $\theta+\zeta=\theta'+\zeta'$ (since $\theta+\zeta+(u-1)=\theta'+\zeta'+(u-1)$).

In order to show that $\theta=\theta'$, $\zeta=\zeta'$ and $\psi=\psi'$, it is convenient to consider the four subcases (a) $\eta>\theta$, $\eta'>\theta'$, (b) $\eta>\theta$, $\eta'=\theta'$, (c) $\eta=\theta$, $\eta'>\theta'$, and (d) $\eta=\theta$, $\eta'=\theta'$.
In (a) we have $\zeta=0=\zeta'$, so $\theta=\theta'$ and $\psi=\psi'$ as required.
In (b), we have $\zeta=0$ hence $\theta=\theta+\zeta=\theta'+\zeta'\GEQ\theta'=\eta'=\eta>\theta$, a contradiction, so this case cannot occur.
Similarly in (c), we have $\theta'=\theta'+\zeta'=\theta+\zeta\GEQ\theta=\eta=\eta'>\theta'$, again a contradiction.
Finally in (d) we have $\theta=\eta=\eta'=\theta'$, hence $\zeta=\zeta'$ and $\psi=\psi'$.

It remains to look at at the last $\psi\,(=\psi')$ columns of $t_{K'}$ and $K$, which contain precisely $u-1$ columns of length 3 in each case.
A similar argument as for the case of the column of length 4 shows that the $u-1$ columns of length 3 in $t_{K'}$ move to the $u-1$ columns of length 3 in $K$.
This completes the proof in the case $s=t$, since $\psi=u-1$ in this case as we have seen.
If $s>t$, counting the number of columns of length 1 in $K$ and $K'$ occurring between the block of $\zeta$ columns each having length 2 and the first column of length 3, we see that the first column of length 3 occurs in exactly the same position in both diagrams.
Similarly, by looking at the number of columns of length 1 between any pair of consecutive columns of length 3 in $K$ and $K'$, we conclude that $K=K'$.
\end{proof}

\begin{lemma}\label{lemma:4.8}
Under the hypothesis and notation of Example~\ref{ex:3.10a}(ii), let $K=N^{(\mathcal S)}$ and $K'=N^{(\mathcal S')}$ be the diagrams corresponding to $\mathcal S=(\eta,\varepsilon,\theta,\varphi,\zeta)$ and $\mathcal S'=(\eta',\varepsilon',\theta',\varphi',\zeta')$, respectively, where $\varphi\GEQ\theta$ and $\varphi'\GEQ\theta'$.
Suppose that the further constraints $\varepsilon=0$ if $\varphi>\theta$ and $\varepsilon'=0$ if $\varphi'>\theta'$ are in force.
If $w_{K'}$ is a prefix of $w_K$, then $K=K'$.
\end{lemma}
\begin{proof}
We assume the hypothesis, and suppose that $w_{K'}$ is a prefix of $w_K$.
Then $t^Kw_{K'}$ is a standard $K$-tableau.
Since this tableau is obtained by moving the entries of each row of $t_{K'}$ to the nodes of $K$ on the same row, keeping the order these entries appear, the column of length 4 in $t_{K'}$ must move to the column of length 4 in $K$ by a similar argument to that in the previous lemma.
It is also clear that the last $u-1$ columns of $t_{K'}$, each having length 3, move to the last $u-1$ columns of $K$.
Considering the second row nodes (resp., third row nodes) lying to the left of the column of length 4 in each diagram, we see that $\varepsilon+\theta=\varepsilon'+\theta'$ (resp., $\eta+\varepsilon=\eta'+\varepsilon'$).
Similarly, considering the second row nodes lying to the right of the column of length 4 in each diagram, we see that $\zeta=\zeta'$ and $\varphi+\zeta=\varphi'+\zeta'$ from which it follows that $\varphi=\varphi'$.
In order to obtain the desired result we will consider the four subcases (a) $\varphi>\theta$, $\varphi'>\theta'$, (b) $\varphi>\theta$, $\varphi'=\theta'$, (c) $\varphi=\theta$, $\varphi'>\theta'$, and (d) $\varphi=\theta$, $\varphi'=\theta'$.
In (a), we get $\varepsilon=0=\varepsilon'$ from the hypothesis, so $\theta=\theta'$ and $\eta=\eta'$.
In (b) we have $\varepsilon=0$, hence $\theta=\theta'+\varepsilon'\GEQ\theta'=\varphi'=\varphi>\theta$, a contradiction, so this case cannot occur.
Similarly, in (c), we have $\varepsilon'=0$, hence $\theta'=\varepsilon+\theta\GEQ\theta=\varphi=\varphi'>\theta'$, again a contradiction.
Finally in (d) we have $\theta=\varphi=\varphi'=\theta'$, so $\varepsilon=\varepsilon'$ and $\eta=\eta'$.
We conclude that $K=K'$.
\end{proof}

\section{Explicit results on  minimal determining sets}\label{sec5}

In this section we obtain explicit descriptions of the minimal determining sets of certain $W$-graph ideals in the symmetric group.
These are (right) $W$-graph ideals corresponding to the Kazhdan-Lusztig cell $\mathfrak C(\lambda)$ for $\lambda=(\lambda_1,\lambda_2,\lambda_3,1^r)$ or $(1^r,\mu_1,\mu_2,\mu_3)$, or to the union of cells obtained by inducing such a Kazhdan-Lusztig cell $\mathfrak C(\lambda)$.

\begin{theorem}
\label{thm:3.13a}
Let $r\GEQ4$ and $s\GEQ t\GEQ u\GEQ 1$.
Let $\lambda=(\lambda_1,\ldots,\lambda_r)$ be a composition where
$\tilde\lambda=(\lambda_1,\lambda_2,\lambda_3)$ is a permutation of
$(s,t,u)$ and $\lambda_i=1$ if $i>3$.
Then $\mathcal E(\lambda)=\mathcal E_s(\lambda)$ if, and only if,  $\lambda_1=s$ or $t$.

Explicit descriptions of the elements of $\mathcal E(\lambda)$ and
$\mathcal E_s(\lambda)$ are given in the proof and the values of
$|\mathcal E_s(\lambda)|$ and $|\mathcal E(\lambda)|-|\mathcal E_s(\lambda)|$ are given in
Tables~\ref{tbl:5a} and~\ref{tbl:6a}.
\end{theorem}
\begin{proof}

Since the results for $r>4$ can be obtained easily by induction on $r$ using Result~\ref{prop:3.2a} assuming the results for $r=4$, we consider the case $r=4$.
In particular, we have that $\lambda$ satisfies Hypothesis~(*).
We suppose further that $d\in Z(\lambda)$ and we let $D=D(d,\lambda)\,(\in\mathcal D^{(\lambda)})$, so $d=w_D$.
Then $D$ is an admissible diagram by Result~\ref{res:8a}(ii).
Below, the cases (I) $\lambda_1=s$ or $\lambda_1=t$, and (II) $\lambda_1=u$ and $t>u$, will be considered separately.

Case (I): $\lambda_1=s$ or $t$:
By Lemma~\ref{lemma:lambda1=st}, we have $D=s(\Pi')$ for some form-A $s$-path $\Pi'$.
Moreover, $d$ is a prefix of $w_{D'}$ for some special diagram $D'\in\mathcal D^{(\lambda)}$ by Remark~\ref{rem_4.13}(ii).
In particular, $Y(\lambda)=Y_s(\lambda)$ if $\lambda_1=s$ or $t$.
Diagram $D'$ has $s$ columns and the nodes on any row of $D'$ of length $s$ have column indices $j$, for $1\LEQ j\LEQ s$.
The node on row 4 of $D'$ is $(4,l)$, where $1\LEQ l\LEQ s$.
Let $A$ be the set of column indices of the nodes of $D'$ of any row of $D'$ of length $t$, and let $B$ be the set of column indices of the nodes of $D'$ on any one of its rows of length $u$.
Then $l\in B\subseteq A\subseteq\{1,\ldots, s\}$.
Clearly, $|A|=t$ and $|B|=u$.
Finally, let $m=\min(B)$.

The case $\lambda$ is a partition is already covered in [MP03, Lemma 3.3] and in this case $\mathcal E^{(\lambda)}$ has a single element , the Young diagram associated to $\lambda$.
It will be convenient, in order to complete case (I), to consider the subcases (i) $\tilde\lambda=(s,u,t)$, (ii) $\tilde\lambda=(t,s,u)$, and (iii) $\tilde\lambda=(t,u,s)$ (even though these subcases are not disjoint and there are also intersections with the partition case already considered).

Subcase (I)(i): $\tilde\lambda=(s,u,t)$:
Let $\delta\colon A\to\{1,\ldots,t\}$ be the order preserving bijection and let $C=B\delta$.
Consider the diagram $E=F_C$ as in Example~\ref{ex:3.8a}(i).
Then $E$ is a special, and hence admissible, diagram such that $t^Ew_{D'}$ is a standard $D$-tableau.
So $w_{D'}$ is a prefix of $w_E$.
Also from the discussion in Example~\ref{ex:3.8a}(i) we can deduce that $\mathcal E^{(\lambda)}$ is precisely the set of diagrams $F_C$ for all the different permitted choices of $C$.
In particular, $|Y(\lambda)|=|Y_s(\lambda)|=\binom{t}{u}$ in this case.
(For example in the case $t=u$ there is a unique choice for $F_C$ and this is the  Young diagram associated to the partition $\lambda$.)

Subcase (I)(ii): $\tilde\lambda=(t,s,u)$: Let $C$ be the set of $u$ smallest indices in $A$ and consider the diagram $E=G_C$ as in Example~\ref{ex:3.8a}(ii).
Clearly, $w_{D'}$ is a prefix of $w_E$.
From the discussion in Example~\ref{ex:3.8a}(ii) we see that $\mathcal E^{(\lambda)}$ consists precisely of the diagrams $G_C$ for all the different permitted choices of $C$. It follows that in this case $|Y(\lambda)|=|Y_s(\lambda)|=\binom{s-t+u}{u}$.

Subcase (I)(iii): $\tilde\lambda=(t,u,s)$: 
We will need to split this case into two subcases
(a) $m\,(=\min(B))\GEQ s-t+1$, and (b) $m<s-t+1$.
If $m\GEQ s-t+1$, we set $C=B$ and $v=m$ (so $v=\min(B)=\min(C)$) and let $\tilde C=C-\{v\}$, (so $\tilde C\subseteq\{s-t+2,\ldots,s\}$).
Comparing with Example~\ref{ex:3.8a}(iii) 
we also let $\tilde v=s-t+1$ (since $v\GEQ s-t-1$).
If $m<s-t+1$, let $\delta\colon A\to\{m\}\cup\{s-t+2,\ldots,s\}$ be the order preserving bijection.
We also set $v=m$, $\tilde C=(B-\{m\})\delta$ and $C=\tilde C\cup\{v\}$.
Hence $v\,(=m)=\min(C)$ and $\tilde C\subseteq\{s-t+2,\ldots,s\}$.
Comparing with Example~\ref{ex:3.8a}(iii), we set $\tilde v=v$ in this subcase.

In either of the subcases (a) or (b) above we get, by setting $E=H_C$, that $t^{E}w_{D'}$ is a standard $E$-tableau, with $E\in\mathcal D^{(\lambda)}$ being a special (and hence admissible) diagram.
Thus, combining with the last paragraph of Example~\ref{ex:3.8a}(iii) we conclude that in the case $\tilde\lambda=(t,u,s)$, the set $\mathcal E^{(\lambda)}$ consists of the diagrams $H_C$ for all the different permitted choices of~$C$.
For the subcase $m\GEQ s-t+1$ there are $\binom{t}{u}$ such diagrams $H_C$, which is the number of $u$-sets in $\{s-t+1,\ldots,s\}$ as these determine $v$ uniquely.
For the subcase $m<s-t+1$ there are $(s-t)\binom{t-1}{u-1}$ such diagrams $H_C$ since, here, we can combine each of the $s-t$  choices of $v$ with each of the $\binom{t-1}{u-1}$ choices of $C-\{v\}$ in $\{s-t+2,\ldots,s\}$.
We conclude that for $\tilde \lambda=(t,u,s)$ we have $|Y(\lambda)|=|Y_s(\lambda)|=(s-t)\binom{t-1}{u-1}+\binom{t}{u}$.
Observe that in the special case $t=s$ (here $s-t+1=1$, so we cannot have $m<s-t+1$), the $\binom{t}{u}$ different diagrams $H_C$ which occur are precisely the $\binom{t}{u}$ different diagrams $F_C$ occurring in subcase (I)(i) with $s=t$.

Case (II): $\lambda_1=u$ and $t>u$ (which is equivalent to $\lambda_1\ne s$ and $\lambda_1\ne t$):
Recall that $d$ denotes an arbitrary element of $Z(\lambda)$ and $D=D(d,\lambda)$, but in this case diagram $D$ may or may not have form-A $s$-paths.
Our upshot is to show that $d$ is a prefix of $w_K$ for some admissible diagram $K$ of shape $M^{(\mathcal S)}$ or $N^{(\mathcal S)}$ (as these are defined in Example~\ref{ex:3.10a}) according as $\tilde\lambda=(u,s,t)$ or $(u,t,s)$.
As an intermediate goal, we aim to show that $d$ is a prefix of $w_E$, for some diagram  $E\in\mathcal D^{(\lambda)}$ satisfying Hypothesis~($\dag$), but also having some additional properties as we will see later.

First, we consider the case $D$ does not have a form-A $s$-path.
Since $D$ is admissible, Lemma~\ref{lem:3.12a} ensures that $D$ has a form-B $s$-path.
Choose one such a form-B $s$-path $\Pi=(\pi_1,\ldots,\pi_s)$ and define the $s$-tuple $\hat a=(a_1,\ldots,a_s)$ by setting $a_j$ to be the length of $\pi_j$ for $1\LEQ j\LEQ s$.
Let $j_1$ be the first $j$ with $a_j=3$ and let $j_2$ be the second such~$j$.

Comparing with Lemma~\ref{lem:3.12a}  we see that $\pi_{j_1}$ is the path of length 3 in $\Pi$ containing the node on row 4 (otherwise $\Pi$ would be equivalent to a form-A $s$-path in $D$).
Next, we form the diagram $\tilde E=D(\Pi)$, see Remark~\ref{rem_4.13}.
From the same remark we see that $t^{\tilde E}d$ is a standard  $\tilde E$-tableau, since $\Pi$ is ordered.
Clearly $\tilde E$ is not admissible since it has no path of length 4.
The $\tilde E$-tableau $t^{\tilde E}d$ has exactly $s$ columns and takes the form given in Table~\ref{tbl:3a}.
In this table we denote by $b_j$ or $c_j$ the entries on rows 2 or 3, accordingly, of column $j$ of $t^{\tilde E}d$ for $1\LEQ j\LEQ s$.
Symbol $\hat C$ denotes the columns with index $j>j_2$; these have entries either on rows 1, 2 and 3, or on rows 2 and 3 or on a single row (row 2 if $\tilde\lambda=(u,s,t)$ or row 3 if $\tilde\lambda=(u,t,s)$).
See Lemma~\ref{lem:3.12a} for the distribution of the nodes of $\tilde E$ in its columns, taking into account the way $\tilde E$ has been constructed.
In particular, $\tilde E$ has no columns of length 1 if $s=t$.

Entries $b_j$ or $c_j$ may be blank if $j\not\in\{j_1,j_2\}$, however $b_j$ (resp., $c_j$) cannot be blank if $\tilde\lambda=(u,s,t)$ (resp., $\tilde\lambda=(u,t,s)$) again by Lemma~\ref{lem:3.12a}.
None of $b_{j_1}$, $b_{j_2}$, $c_{j_1}$, $c_{j_2}$ is blank from our hypothesis that they belong to columns of length 3.
Entry $y$ is not blank and it is the sole entry on row 4 of $t^{\tilde E}d$.
Moreover, $x$ is not blank and it is the first entry on row 1 of $t^{\tilde E}d$.
Hence, $y$ is the entry occupying the node $(4,j_1)$ of $\tilde E$ and $x$ is the entry occupying the node $(1,j_2)$ of $\tilde E$, with $j_1<j_2$.

\begin{table}[h]
\centering
$
\setlength{\arraycolsep}{0pt}
\setlength{\fboxsep}{0pt}
 t^{\tilde E}d =
\fbox{
$
\begin{array}{ccccccc}
{
\setlength{\arraycolsep}{3pt}
 \begin{array}{cccccccc}
  &  &   &  &  & &  &  x
 \\
 b_1 & \cdots & b_{j_1} & b_{j_1+1} & \cdots & b_{j'} & \cdots & b_{j_2}
 \\
 c_1 & \cdots & c_{j_1} & c_{j_1+1} & \cdots & c_{j'} & \cdots & c_{j_2}
 \\
  &  & y &  &  &  & &
 \mbox{}
 \end{array}
}
&
\begin{array}{l}
\makebox[0.75in]{
 $
 \begin{array}{c}
 \hat C\rule{0pt}{0.28in}
 \\
 \rule{0pt}{0.11in}
 \end{array}
 $
}
\\
\makebox[0.75in]{
\rule{0pt}{0.15in}
}
\end{array}
\end{array}
$
}
$
\caption{$t^{\tilde E}d$}
\label{tbl:3a}
\end{table}

The fact that $D$ has no form-A $s$-paths ensures that $x>b_{j_1}$ and $y<c_{j_2}$.
(If $x<b_{j_1}$ this would mean that $x$ lies in a column of tableau $t_D$ of weakly smaller  index than the index of the column containing $b_{j_1}$, which would mean in turn that $D$ has a form-$A$ $s$-path.
We can exclude the possibility $y>c_{j_2}$ by using similar argument.)

Now let $N_1$ be the node in $D$ with $N_1t_D=x$.
Clearly $N_1$ is the first node on row 1 of $D$ from the way $t^{\tilde E}d$ is related to $t_D$.
Since $D$ is admissible, it contains a path $\pi$ of length 4 with $N_1\in s(\pi)$.
Also let $N_i$, with $N_i$ on row $i$ of $D$ for $2\LEQ i\LEQ 4$, be the remaining nodes of $\pi$.
So $s(\pi)=\{N_1,N_2,N_3,N_4\}$ with $N_1t_D=x$ and $N_4t_D=y$
(the last equality follows from the fact that $N_4$ is the only node on row 4 of $D$).
We also have that $N_2t_D=b_k$ and $N_3t_D=c_l$ for some $k,l$ with $1\LEQ k,l\LEQ s$.
Since $\pi$ is a path in $D$, it follows from the way $t_D$ is constructed, that $x<b_k<c_l<y$.
Hence, the relations $b_{j_1}<x$ and $y<c_{j_2}$ obtained above, together with the standardness of $t^{\tilde E}d$,  give that $j_1<k$ and $l<j_2$.

\medskip
{\bf Claim.} There exist $j'$, $j''$ with $j_1\LEQ j''\LEQ j'\LEQ j_2$ and $x<b_{j'}<c_{j''}<y$.

\smallskip
{\bf Proof of Claim.}
Recall that $x<b_k<c_l<y$ and also that $b_{j_1}<x$ where $j_1<k$ and $y<c_{j_2}$ where $l<j_2$.

Suppose first that $k\LEQ l$.
Then $j_1<k\LEQ l<j_2$.
Now at least one of the entries $b_l$ or $c_k$ is non-empty according as $\tilde\lambda=(u,s,t)$ or $(u,t,s)$.
The claim is now proved, using the standardness of $t^{\tilde E}d$, by setting $j'=j''=l$ or $j'=j''=k$ accordingly.
[If $b_l$ is non-empty, then $b_k\LEQ b_l<c_l$ since $k\LEQ l$ and $t^{\tilde E}d$ is standard.
So $x<(b_k\LEQ)\,b_l< c_l<y$ and $j_1<\,(k\LEQ)\, l<j_2$.
If $c_k$ is non-empty, then $x<b_k<c_k\,(\LEQ c_l)<y$ and $j_1<k\,(\LEQ l)<j_2$.]

Suppose now that $k>l$.
We will consider separately the four subcases
(a) $k\LEQ j_2$ and $l\GEQ j_1$,
(b) $k\LEQ j_2$ and $l<j_1$,
(c) $k> j_2$ and $l\GEQ j_1$, and
(d) $k> j_2$ and $l<j_1$.

In subcase (a), we have $j_1\LEQ l<k\LEQ j_2$.
Since $x<b_k<c_l<y$, the claim is proved by setting $k=j'$ and $l=j''$.

In subcase (b), we have $l<j_1<k\LEQ j_2$, so $x<b_k<c_l<c_{j_1}<y$, again using the standardness of $t^{\tilde E}d$.
By setting $j''=j_1$ and $j'=k$, we see that $j_1\LEQ j''<j'\LEQ j_2$ and $x<b_{j'}<c_{j''}<y$ as required.

In subcase (c), we have $j_1\LEQ l<j_2<k$, so by setting $j'=j_2$ and $j''=l$ we get $j_1\LEQ j''<j'=j_2$ and $x<b_{j_2}\,(=b_{j'})<b_k<c_l\,(=c_{j''})<y$ as required.

Finally, in subcase (d), we have $l<j_1<j_2<k$ so $b_{j_2}<b_k$ and $c_l<c_{j_1}$.
Moreover, using the standardness of $t^{\tilde E}d$, we get $x<b_{j_2}<b_k<c_l<c_{j_1}<y$.
By setting $j'=j_2$ and $j''=j_1$, we have $j_1\LEQ j''<j'\LEQ j_2$ and $x<b_{j'}<c_{j''}<y$, thus completing the proof of the claim.

Now let $\hat E\in\mathcal D^{(\lambda)}$ be the underlying diagram of the tableau $t$ obtained by moving the entries in tableau $t^{\tilde E}d$ according to the scheme in Table~\ref{tbl:4a}, where any blank column, that is one corresponding to a blank $c_j$ (with $j''<j\LEQ j'$) or $b_j$ (with $j''\LEQ j<j'$) according as $\tilde\lambda=(u,s,t)$ or $\tilde\lambda=(u,t,s)$, is removed.

\begin{table}[h]
\centering
$
\setlength{\arraycolsep}{0pt}
\setlength{\fboxsep}{0pt}
 t =
\fbox{
$
\begin{array}{ccccccc}
{
\setlength{\arraycolsep}{3pt}
 \begin{array}{cccccccccccccc}
     &        &           &         &           &        &  x
 \\
 b_1 & \cdots & b_{j''-1} & b_{j''} & b_{j''+1} & \cdots & b_{j'}  &           &        &        & b_{j'+1} & \cdots & b_{j_2}
 \\
 c_1 & \cdots & c_{j''-1} &         &           &        & c_{j''} & c_{j''+1} & \cdots & c_{j'} & c_{j'+1} & \cdots & c_{j_2}
 \\
     &        &           &         &           &        &  y      &           &        &        &          &        &
 \mbox{}
 \end{array}
}
&
\begin{array}{l}
\makebox[0.75in]{
 $
 \begin{array}{c}
 \hat C\rule{0pt}{0.28in}
 \\
 \rule{0pt}{0.11in}
 \end{array}
 $
}
\\
\makebox[0.75in]{
\rule{0pt}{0.15in}
}
\end{array}
\end{array}
$
}
$
\caption{$t=t^{\hat E}d$}
\label{tbl:4a}
\end{table}

Clearly, from the construction, $t=t^{\hat E}d$.
Moreover, $\hat E$ satisfies Hypothesis~($\dag$) and $t$ is a standard $\hat E$-tableau.
So, by setting $E=\hat E$, the intermediate goal of showing that $d$ is a prefix of $w_{E}$ for some diagram $E\in\mathcal D^{(\lambda)}$ satisfying Hypothesis~($\dag$) has been achieved in the case $D$ has no form-A $s$-path.
So we assume now that $D$ has a form-A $s$-path, say $\check\Pi$, and we aim to show that $D(\check\Pi)$ can be transformed into a diagram $\check E\in\mathcal D^{(\lambda)}$ satisfying Hypothesis~($\dag$) and which also satisfies the additional requirement that $d$ is a prefix of $w_{\check E}$.
First observe that $D(\check\Pi)$ is special and $t^{D(\check\Pi)}d$ is a standard $D(\check\Pi)$-tableau by Remark~\ref{rem_4.13}.
The construction of $\check E$ from $D(\check\Pi)$ is as follows:
If the column of length 4 in $D(\check\Pi)$ lies to the left of all columns of length 3 in $D(\check\Pi)$, set $\check E=D(\check\Pi)$.
Otherwise, let $\check E$ be the diagram obtained from $D(\check\Pi)$ by moving the single node on row 4 of $D(\check\Pi)$ to the first column of $D(\check\Pi)$ having length 3 (keeping this node on row~4).
In either case diagram $\check E$ satisfies Hypothesis~($\dag$) and $t^{\check E}d$ is a standard $\check E$-tableau by its construction.

For the rest of the proof we denote by $E$ any diagram of shape $\hat E$ or $\check E$ which has been obtained from $D$ via any of the above processes.
In particular, irrespective of whether $E$ is of `type $\hat E$' or of `type $\check E$', diagram $E$ satisfies Hypothesis~($\dag$) and $d$ is a prefix of $w_E$ since $t^Ed$ is standard.
Moreover, the columns of length 1 in $E$ (excluding the regions in $\hat E$ containing the nodes corresponding to the entries $b_j$ in tableau $t$ for $j''\LEQ j<j'$ and entries $c_j$ in tableau $t$ for $j''<j\LEQ j'$ --- these regions can be considered to be `empty' in $\check E$) have their single node on row 2 if $\tilde\lambda=(u,s,t)$, and on row 3 if $\tilde\lambda=(u,t,s)$.
We can also observe that $E$ has $s'$ columns, where $s'=s$ in the case $E$ is of `type $\check E$', and $s'=s+n_2$ where $n_2=|\{j\colon a_j=2$ and $j''<j\LEQ j'\}|$
(resp., $n_2=|\{j\colon a_j=2$ and $j''\LEQ j< j'\}|$) if $\tilde\lambda=(u,s,t)$ (resp., $\tilde\lambda=(u,t,s)$) in the case $E$ is of `type $\hat E$'.
(Note that in the special case $s=t$ we have $n_2=j'-j''$.)

Let $\hat\alpha_D=(\hat\alpha_1,\ldots,\hat\alpha_{s'})$ be the determining tuple for diagram $E$.
For the rest of the proof it will be convenient to consider the subcases $\tilde\lambda=(u,s,t)$ and $\tilde\lambda=(u,t,s)$ separately.
Operations (C1)--(C5) discussed in Remark~\ref{rem:operations} will play a key role.

Subcase (II)(i): $\tilde\lambda=(u,s,t)$:
Applying operations of type (C1) to the columns in the region of $E$ which lies to the left of the (unique) column of length 4 in $E$, and also operations of types (C2) and (C3) to the columns in the region of $E$ which lies to the right of the column of length 4, we see that $w_E$ is a prefix of $w_{K'}$ for some diagram $K'\in\mathcal D^{(\lambda)}$ satisfying Hypothesis~($\dag$) which can be described as follows:
Diagram $K'$ has $s'$ columns (same number of columns as $E$) and its determining tuple $\hat\alpha_{K'}$  begins with an $\varepsilon'$-tuple of 2's, followed by an $\eta'$-tuple of 1's, then has a single 4 and, following the 4, it has a $\theta'$-tuple of  $\bar1$'s ($\theta'=0$ if $E$ is of `type $\check E$'), followed by a $\zeta'$-tuple of 2's, followed by a $\psi'$-tuple of 1's and 3's containing exactly $(u-1)$ 3's.
From the construction, we see that $\eta'\GEQ\theta'$.
[This is obvious if $E$ is of `type $\check E$' since $\theta'=0$ in this case.
If $E$ is of `type $\hat E$', observe that $\eta'\GEQ j'-j''\GEQ\theta'$, since by Lemma~\ref{lem:3.12a} none of the $b_j$'s for $j''\LEQ j<j'$ is empty, whereas we could possibly have some empty $c_j$'s for $j''<j\LEQ j'$.]
If $\eta'=\theta'$ or $\zeta'=0$, we set $K=K'$.
If $\eta'>\theta'$ and $\zeta'>0$, we apply operation (C5) to the first column of length 2 occurring from the left in the block of $\zeta'$ columns of length 2, and then by repeated applications of operation (C1) we can `carry' the column of length 1 (with a single node on row 2) which has resulted from the application of (C5), to the position immediately to the left of the block of $\psi'$ columns of length 1 or 3.
Next, we set $\zeta'(0)=\zeta'$, $\theta'(0)=\theta'$ and, for $k\GEQ1$, $\zeta'(k)=\zeta'(k-1)-1$ and $\theta'(k)=\theta'(k-1)+1$, where $k$ is the number of repetitions of the above routine.
The process stops after $r$ repetitions of the routine, where $r$ is the smallest integer such that either $\zeta'(r)=0$ or $\eta'=\theta'(r)$ and
we let $K$ be the diagram obtained from $K'$ at this stage of the process.
Clearly $K$ has $s'+r$ columns and $\zeta'(r)=0$ if $\eta'>\theta'(r)$.
Moreover $K$ satisfies Hypothesis~($\dag$) and $d$ is a prefix of $w_K$ (see Remark~\ref{rem:operations}).
Comparing with Example~\ref{ex:3.10a}(i) we see that $K=M^{(\mathcal S)}$ for some permitted tuple $\mathcal S=(\varepsilon,\eta,\theta,\zeta,\psi,\mathcal C)$ with $\eta\GEQ\theta$ and the further constraint $\zeta=0$ if $\eta>\theta$.
(In fact we have $\varepsilon=\varepsilon'$, $\eta=\eta'$, $\theta=\theta'(r)=\theta'+r$, $\zeta=\zeta'(r)=\zeta'-r$, $\psi=\psi'+r$.)
Also recall from Example~\ref{ex:3.10a}(i) that for such tuples $\mathcal S$, diagram $M^{(\mathcal S)}$ is an admissible diagram and, in addition, $M^{(\mathcal S)}$ is special if, and only if, $\theta=0$.
It now follows from Lemma~\ref{lemma:4.7} that $\mathcal E^{(\lambda)}$ is precisely the set of diagrams $M^{(\mathcal S)}$ where the tuple $\mathcal S$ of non-negative integers satisfies the above constraints $\eta\GEQ\theta$ and $\zeta=0$ if $\eta>\theta$.
(Clearly $\mathcal E_s^{(\lambda)}$ is the subset of $\mathcal E^{(\lambda)}$  obtained by imposing the further restriction $\theta=0$.)

Counting nodes on the second and third rows, we get
$s=\varepsilon+\eta+1+\zeta+\psi$ and
$t=\varepsilon+\theta+\zeta+u$.
So, $\theta+\zeta\LEQ t-u$ and $\psi=s-t-(u-1)-\bar{\eta}$,
where $0\LEQ\bar{\eta}=\eta-\theta\LEQ s-t$.
Thus, by setting $v=s-t+u$, we see that for $M^{(\mathcal S)}$ to belong to $\mathcal E^{(\lambda)}-\mathcal E_s^{(\lambda)}$ (resp., for $M^{(\mathcal S)}$ to belong to $\mathcal E_s^{(\lambda)}$), the number of permitted determining tuples with $\zeta\GEQ1$ is  
\[
\sum_{\theta=1}^{t-u-1} \sum_{\zeta=1}^{t-u-\theta}\binom{v-1}{u-1}
=\binom{t-u}{2}\binom{v-1}{u-1}
\quad
\Big(\mbox{resp.,\ } \sum_{\zeta=1}^{t-u}\binom{v-1}{u-1}=(t-u)\binom{v-1}{u-1} \Big)
\]
and the number of permitted determining tuples with $\zeta=0$ is    
\[
\sum_{\bar{\eta}=0}^{s-t} \sum_{\theta=1}^{t-u}\binom{v-1-\bar{\eta}}{u-1}
=(t-u)\binom{v}{u}
\quad
\Big(\mbox{resp.,\ }\sum_{\bar{\eta}=0}^{s-t}\binom{v-1-{\eta}}{u-1}=\binom{v}{u} \Big).
\]

Thus, we have determined $|\mathcal E_s(\lambda)|$ and $|\mathcal E(\lambda)-\mathcal E_s(\lambda)|$ to be the
values given in Tables~\ref{tbl:5a} and~\ref{tbl:6a} and the corresponding
diagrams take the form $M^{(\mathcal{S})}$
where $\mathcal{S}=
(\varepsilon,\eta,\theta,\zeta,\psi,\mathcal{C})$
and $\mathcal{C}$ denotes an arbitrary set of $u-1$ columns among
the last $\psi$ columns.

Subcase (II)(ii): $\tilde\lambda=(u,t,s)$:
We use similar arguments as for subcase (II)(i) but this time we begin by first applying a sequence of operations from types (C2), (C3) and (C4) to the columns lying in the region of $E$ which is to the right of the column of length 4 and a sequence of operations of type (C3) to the columns of $E$ lying in the region to the left of the column of length 4, in order to obtain a diagram $K'\in\mathcal D^{(\lambda)}$ with the following properties:
Diagram $K'$ satisfies Hypothesis~($\dag$), the $K'$-tableau $t^{K'}d$ is standard, $K'$ has exactly $s'$ columns and the determining tuple $\hat\alpha_{K'}$ of $K'$ begins with an $\eta'$-tuple of $\bar1$'s, followed by an $\varepsilon'$-tuple of 2's, followed by a $\theta'$-tuple of 1's, followed by a single 4, and following the 4, a $\varphi'$-tuple of $\bar1$'s followed by a $\zeta'$-tuple of 2's, followed by a $(u-1)$-tuple of 3's.
From the construction of $E$ and the types of operation used to obtain $K'$ from $E$, we also see that $\varphi'\GEQ j'-j''\GEQ\theta'$, if $E$ is `of type $\hat E$'.
In the case $E$ is `of type $\check E$', the relation $\varphi'\GEQ\theta'$ holds trivially since $\theta'=0$ in this case.

Finally, in a similar fashion as in case (II)(i), now applying operations of types (C5) and (C3) to the columns of $K'$ corresponding to the $\varepsilon'$-tuple of 2's which lie in the region to the left of the column of length 4 (but working from right to the left on this block of columns) we obtain a diagram $K$ from $K'$ with $w_{K'}$ a prefix of $w_K$ (hence with $d$ a prefix of $w_K$) such that
$K=N^{(\mathcal S)}$ for some tuple $\mathcal S=(\eta,\varepsilon,\theta,\varphi,\zeta)$ of non-negative integers satisfying the further constraints $\varphi\GEQ\theta$ and $\varepsilon>0$ if $\varphi>\theta$.
(Compare with Example~\ref{ex:3.10a}(ii).)
Lemma~\ref{lemma:4.8}  now ensures that the set  $\mathcal E^{(\lambda)}$ is precisely the set of diagrams  $N^{(\mathcal S)}$ for which the conditions $\varphi\GEQ\theta$ and $\varepsilon>0$ if $\varphi>\theta$ are satisfied by $\mathcal S$.
For the subset $\mathcal E^{(\lambda)}_s$ we need the further restriction $\theta=0$ since from Example~\ref{ex:3.10a}(ii) we know that $N^{(\mathcal S)}$ is special if, and only if, $\theta=0$.

Counting nodes on the second and third rows, we get
$t=\varepsilon+\theta+\zeta+u$ and
$s=\eta+\varepsilon+\varphi+\zeta+u$.
So, $\theta+\zeta\LEQ t-u$ and $\bar{\varphi}=s-t-\eta$
where $\bar{\varphi}=\varphi-\theta$.
Given $\theta$ and $\bar{\varphi}$
with $0\LEQ\theta\LEQ t-u$ and $0\LEQ\bar{\varphi}\LEQ s-t$, the
quantities $\eta$, $\varphi$ and $\varepsilon+\zeta$ are determined.
If additionally, $\bar{\varphi}>0$ then $\varepsilon=0$, so $\zeta$
is determined,
whereas if $\bar{\varphi}=0$ then $0\LEQ\varepsilon\LEQ t-u-\theta$.
Thus, for $N^{(\mathcal S)}$ to belong to $\mathcal E^{(\lambda)}-\mathcal E_s^{(\lambda)}$ (resp., for $N^{(\mathcal S)}$ to belong to $\mathcal E_s^{(\lambda)}$) the number of permitted determining tuples with $\bar{\varphi}\GEQ1$
is
\[
\sum_{\theta=1}^{t-u} \sum_{\bar{\varphi}=1}^{s-t} 1
=(t-u)(s-t)
\qquad \Big(\mbox{resp.,\ } \sum_{\varphi=1}^{s-t}1=s-t\Big)
\]
and the number of permitted deternining tuples with $\bar{\varphi}=0$ is
\[
\sum_{\theta=1}^{t-u} \sum_{\varepsilon=0}^{t-u-\theta} 1
=\binom{t-u+1}{2}
\qquad \Big(\mbox{resp.,\ } \sum_{\varepsilon=0}^{t-u} 1=t-u+1\Big).
\]

Thus, we have determined $|\mathcal E_s(\lambda)|$ and $|\mathcal E(\lambda)-\mathcal E_s(\lambda)|$ to be the
values given in Tables~\ref{tbl:5a} and~\ref{tbl:6a} and the corresponding
diagrams take the form $N^{(\mathcal{S})}$
where $\mathcal{S}= (\eta,\varepsilon,\theta,\varphi,\zeta)$.
\end{proof}
\par
\begin{table}[h]
\centering
$
\setlength{\arraycolsep}{4pt}
\begin{array}{ccccccc}
\tilde\lambda & (s,t,u) & (s,u,t) & (t,s,u)
& (t,u,s) & (u,s,t) & (u,t,s)
\\\hline
\\[-1ex]
& 1 & \dbinom{t}{u} & \dbinom{s-t+u}{u}
& \begin{array}{c}
(s-t)\dbinom{t-1}{u-1}\\[2ex]
+\dbinom{t}{u}
\end{array}
& \begin{array}{c}
(t-u)\dbinom{s-t+u-1}{u-1}\\[2ex]
+\dbinom{s-t+u}{u}
\end{array}
& s-u+1
\end{array}
$
\caption{
Values of $|\mathcal E_s^{(\lambda)}|$,
$s\GEQ t\GEQ u$, $r>3$. (Theorem~\ref{thm:3.13a})
}
\label{tbl:5a}
\end{table}
\par
\begin{table}[h]
\centering
$
\setlength{\arraycolsep}{8pt}
\begin{array}{ccccccc}
\tilde\lambda & (u,s,t) & (u,t,s)
\\\hline
\\[-1ex]
&
\dbinom{t-u}{2}\dbinom{s-t+u-1}{u-1}
+(t-u)\dbinom{s-t+u}{u}
&
(t-u)(s-t)+\dbinom{t-u+1}{2}.
\\[1ex]
\end{array}
$
\caption{
Values of $|\mathcal E^{(\lambda)}-\mathcal E_s^{(\lambda)}|$,
$s\GEQ t\GEQ u$, $r>3$. (Theorem~\ref{thm:3.13a})
}
\label{tbl:6a}
\end{table}

\begin{remark}\label{Rem5.2}
(i) In view of Result~\ref{res:10a}, we immediately get from Theorem~\ref{thm:3.13a} (and Tables~\ref{tbl:5a} and~\ref{tbl:6a}) all the corresponding information about the set~$\mathcal E^{(\mu)}$ where the composition $\mu$ has the form $(1^r,\mu_1,\mu_2,\mu_3)$.

(ii)
Let $\lambda$ be a composition of $n$ with $r$ parts.
Recall that $Z(\lambda)$ is a right ideal in $S_n=\langle s_1,\ldots,s_{n-1}\rangle$ and that $Z(\lambda)\hat{\mathfrak X}$ is a right ideal in $S_{n+1}=\langle s_1,\ldots,s_n\rangle$, where $s_i$ is the basic transposition $(i\ i+1)$, and $\hat{\mathfrak X}$ denotes the set of distinguished right coset representatives of $S_n$ in $S_{n+1}$.
The longest element of $\hat{\mathfrak X}$ is the element $s_ns_{n-1}\ldots s_1\,(=(1\ 2\ldots n+1)$ in cycle-notation).
Given $D\in\mathcal E^{(\lambda)}$, we define $\hat D=\{(r+1,1)\}\cup\{(i,j+1)\colon (i,j)\in D\}$, so $\hat D$ is a diagram of size $n+1$.
By Proposition~\ref{prop:3.12}, the minimal determining set of the right ideal $Z(\lambda)\hat{\mathfrak X}$ in $S_{n+1}$ is the set $\{w_{\hat D}\colon D\in\mathcal E^{(\lambda)}\}$.
(Comparing with the discussion in Section~\ref{subsec_sym_gr}, we can consider this set to be the rim of the induced union of cells $w_JZ(\lambda)\hat{\mathfrak X}$.)
Thus, the explicit results in~\cite{MPa17}, \cite{MPa21} and also in Section~\ref{sec5} of the present paper on the minimal determining sets of various families of right ideals of the form $Z(\lambda)$ lead to an explicit description of the minimal determining sets of the corresponding induced right ideals $Z(\lambda)\hat{\mathfrak X}$.

(iii)
The results of this paper together with the results of~\cite{MPa17} and~\cite{MPa21} give complete information about the set~$\mathcal E^{(\lambda)}$ for all compositions $\lambda$ of $n$ for $n\LEQ 6$.
Using similar methods we have also completed the case $n=7$.
More detailed information about the sizes of the rims of the corresponding cells in the form of tables can be obtained from any of the authors on request.
\end{remark}

\end{section}

\end{document}